\documentclass[10pt,english]{article}
\usepackage[utf8]{inputenc}
\usepackage[T1]{fontenc}
\usepackage{textcomp}
\usepackage[english]{babel}
\usepackage{url}
\usepackage{amsmath,amssymb,amsthm}
\usepackage{tabularx}

\usepackage{graphicx}
\usepackage{mathrsfs}

\usepackage{import}
\usepackage[dvipsnames]{xcolor}
\usepackage[colorlinks=true,linktoc=page]{hyperref}
\hypersetup{urlcolor=NavyBlue, citecolor=Red, linkcolor=RubineRed}

\usepackage[all,cmtip]{xy}
\usepackage{manfnt}

\usepackage[backend=biber, style=alphabetic, natbib=true,hyperref=true]{biblatex}
\theoremstyle{plain}
\newtheorem{Theo}{Theorem}[section] 
\newtheorem{pro}[Theo]{Proposition}        
\newtheorem{lem}[Theo]{Lemma}            
\newtheorem{cor}[Theo]{Corollary}

\theoremstyle{definition}
\newtheorem{defi}[Theo]{Definition}
\newtheorem{exem}[Theo]{Example}
\newtheorem{nota}[Theo]{Notation}

\newtheorem{hyp}[Theo]{Hypothesis}

\theoremstyle{remark}
\newtheorem{rem}[Theo]{Remark}

\def\ogg~{{\rm \og}}   

\def\nn{\noindent}

\def\q{\nn}
\def\qq{\nn\quad}
\def\qqq{\nn\quad\quad}

\def\emptyset{\varnothing}

\def\NN{{\mathbb N}}    
\def\ZZ{{\mathbb Z}}     
\def\RR{{\mathbb R}}    
\def\QQ{{\mathbb Q}}    
\def\CC{{\mathbb C}}    
\def\AA{{\mathbb A}}    
\def\PP{{\mathbb P}}

   \def\cM{{\mathcal M}}      \def\cT{{\mathcal T}}    \def\cO{{\mathcal O}}              \def\cL{{\mathcal L}}     

 \def\mfp{{\mathfrak p}}                           

\newcommand\ct{\operatorname{cotan}}         

\newcommand{\Ker}{\operatorname{Ker}}

\newcommand{\disf}[2]{D^+(#1,#2)}
\newcommand{\diso}[2]{D^-(#1,#2)}
\newcommand{\fdisf}[2]{\cO(D^+(#1,#2))}

\newcommand{\h}[1]{\mathscr{H} (#1)}
\newcommand{\A}[1]{\AA^{1,an}_{#1}}
\newcommand{\Po}[1]{\PP^{1,an}_{#1}}

\newcommand{\Ct}[1]{\hat{\otimes}_{#1}}
\newcommand{\LL}[2]{\cL_{#1}(#2)}

\newcommand{\Lk}[1]{\LL{k}{#1}}
\newcommand{\nsp}[1]{\rVert #1\rVert_{Sp}}
\newcommand{\Nsp}[2]{\rVert #2\rVert_{Sp,#1}}
\newcommand{\piro}[2]{\pi_{#1/#2}}
\newcommand{\pik}[1]{\piro{#1}{k}}

\newcommand{\nor}[1]{\rVert #1\rVert}

\newcommand{\discf}[3]{D^+_{#1} (#2,#3)}
\newcommand{\disco}[3]{D^-_{#1} (#2,#3)}
\newcommand{\fdiscf}[3]{\cO(D^+_{#1} (#2,#3))}

\newcommand{\couf}[3]{C^+ (#1,#2,#3)}
\newcommand{\couo}[3]{C^- (#1,#2,#3)}
\newcommand{\fcouf}[3]{\cO(C^+ (#1,#2,#3))}

\newcommand{\Couf}[4]{C^+_{#1} (#2,#3,#4)}
\newcommand{\Couo}[4]{C^-_{#1} (#2,#3,#4)}

\newcommand{\Pro}[2]{\pi_{#2/#1}}

\newcommand{\Uni}[2]{\sigma_{#2/#1}}
\newcommand{\uni}[1]{\Uni{k}{#1}}

\newcommand{\Fonction}[4]{\begin{array}[c]{rcl} 
                            #1&\longrightarrow&#2\\ #3&\mapsto&
                                                                #4\\ \end{array}}

\newcommand{\fra}[1]{\frac{1}{#1}}

\renewcommand\phi{\varphi}

\def\ct{\hat{\otimes}}
\def \ito{\hookrightarrow}
\def \ot{\otimes}
\def \<{\langle}
\def \>{\rangle}

\def\|{\rVert}
\def\*{\blacklozenge}

\def\b{\mathbf}

\def\dT{\frac{d}{dT}}
\def\d{\frac{d}{dS}}

\def \R+{\RR_{+}}

\def\Ak{\A{k}}
\def \Ac{\A{\CC}}

\def \Pk{\Po{k}}

\def \-1{^{-1}}

\def \Hx{\mathscr{H}(x)}

\def \rk{\tilde{k}}
\def \crk{\text{char}(\rk)}

\def\Bigcup{\bigcup\limits}

\def\Lim{\lim\limits}

\def\Sup{\sup\limits}

\def\Min{\min\limits}

\def\(({(\!(}
\def\)){)\!)}

\def\MM{\mathscr{M}}

\hypersetup{
pdfauthor = {Tinhinane Amina AZZOUZ},
pdfsubject = {Thesis},
pdftitle = {Thesis},
pdfkeywords = {thesis},
}

\title{Spectrum of a linear differential equation with constant coefficients}
\date {\today}
\author{Tinhinane A. \textsc{AZZOUZ}}

\begin{document}

\maketitle

\begin{abstract}
In this paper we compute the spectrum, in the sense of Berkovich, of an
ultrametric linear differential equation with constant coefficients, defined over an affinoid
domain of the analytic affine line $\Ak$. We show that it is a finite
union of either closed disks or  topological closures of open
disks and that it satisfies a continuity property.
  
\end{abstract}


\section{Introduction}

Differential equations constitute an important tool for the
investigation of algebraic and analytic varieties, over the
complex and the $p$-adic numbers. 
Notably, de Rham cohomology is one of the most powerful way to obtain
algebraic and analytic informations. Besides, ultrametric phenomena appeared naturally studying formal
Taylor solutions of the equation around singular and regular
points.    

As soon as the theory of ultrametric differential equations became a
central topic of investigation  around 1960, after the work of
B. Dwork, P. Robba (et al.), the following
intersting phenomena appeared. 
In the ultrametric setting, the solutions of a linear differential
equation may fail to converge as expected, even if the coefficients of
the equation are entire functions. For example, over the ground field
$\QQ_p$ of $p$-adic numbers, the exponential power series
$\exp(T)=\sum_{n\geq 0}\frac{T^n}{n!}$ which is  solution of the
equation $y'=y$ has radius of convergence equal to $|p|^{\fra{(p-1)}}$
even though the equation shows no singularities. However, the behaviour of the radius of convergence is
well controlled, and 
its knowledge permits to obtain several informations about the
equation. Namely it controls the \emph{finite dimensionality} of the de Rham cohomology. For more details, we refer the reader to the recent work of
J. Poineau and A. Pulita \cite{and}, \cite{np2}, \cite{np3} and
\cite{np4}.

The starting point of this paper is an interesting relation between
this radius and the notion of \emph{spectrum} (in the
sense of Berkovich). Before discussing more in detail our results,
we shall quickly explain this relation. 

Consider a ring $A$ together with a derivation
$d:A\to A$. A \emph{differential module} on $(A,d)$ is a finite free
$A$-module $M$ together with a $\ZZ$-linear map
$$\nabla: M\to M$$
satisfying for all $m\in M$ and all $f\in A$ the relation
$\nabla(fm)=d(f)m+f\nabla(m)$. If an isomorphism $M\cong A^\nu$ is
given, then $\nabla$ coincides with an operator of the form
$$d + G : A^\nu\to A^\nu$$
where $d$ acts on $A^\nu$ component by component and  $G$ is a square matrix with coefficients in $A$.

If $A$ is moreover a Banach algebra with respect to a given norm
$\nor{.}$ and  $A^\nu$ is endowed with the max norm, then we can endow the
algebra of continuous $\ZZ$-linear endomorphisms with the operator norm, that
we still denote $\nor{.}$. The
spectral norm of $\nabla$ is given by $$
\nsp{\nabla}=\Lim_{n\to \infty}\nor{\nabla^n}^{\fra{n}}.$$

The link between the radius of convergence of $\nabla$ and its
spectrum is then the following:  on the one hand, when $\nor{.}$ is a
Berkovich point of type other than $1$ of the affine line and $A$ is
the field of this point, the spectral norm
$\nsp{\nabla}$ coincides
with the inverse of the radius of convergence of $M$ at $\nor{.}$ multiplied by some
constant (cf. \cite[p. 676]{dwork} or \cite[Definition 9.4.4]{Ked}).
On the other hand the spectral norm $\nsp{\nabla}$ is also equal to the radius of the smallest disk centred at zero and containing
the spectrum of $\nabla$ in the sense of Berkovich (cf. \cite[Theorem
7.1.2]{Ber}). We refer to the introduction of
Section~\ref{sec:spectr-assoc-diff-4} for a more precise explanation.

The spectrum appears then as a new invariant of the connection
$\nabla$ generalizing and refining the radius of convergence. This
have been our  first motivation, however the study of the spectrum of
an operator has its own interest and it deserves its own independent theory. 

In this
paper we focus 
on the computation of the spectrum of $\nabla$
in the case where $\nabla$ has constant
coefficients, which means that there is a basis of $M$ in which the
matrix $G$ has coefficients in the base field of constants (i.e. the field of elements killed by the derivations).
We show that when we change the domain of definition of
the equation  the spectrum has a uniform behaviour: it is
a finite union of either  closed disks or
topological closures (in the sense of Berkovich) of open disks.

Let $k$ be an arbitrary field and $E$ be a $k$-algebra with
unit. Recall that classically, the spectrum $\Sigma_f(E)$ of an
element $f$ of $E$ (cf. \cite[\S 1.2. Defintion 1]{bousp}) is the set
of elements $\lambda$ of $k$ such that $f-\lambda.1_E$ is not
invertible in $E$. In the case where $k=\CC$ and $(E,\nor{.})$ is a
$\CC$-Banach algebra, the spectrum of $f$ satisfies the following
properties:
\begin{itemize}
\item It is not empty and compact.
\item The smallest disk centred at 0 and containing $\Sigma_f(E)$ has
  radius equal to $\nsp{f}=\Lim_{n\to+\infty}\nor{f^n}^{\fra{n}}$.
\item The resolvent function $R_f:\CC\setminus\Sigma_f\to E$, $\lambda\mapsto
  (f-\lambda.1_E)\-1$ is an analytic function with values in $E$.
\end{itemize}
Unfortunately, this may fail in the ultrametric case. In \cite{Vi85} M. Vishik provides an
example of operator with empty spectrum and with a resolvent which
is only locally analytic. We  illustrate this pathology with an
example in
our context, where connections with empty classical spectrum abound.

\begin{exem}
  Consider the field $k=\CC$ endowed with the trivial
absolute value. Let $A:=\CC\((S\))$ be the field of Laurent power
series
endowed with the $S$-adic absolute value given by $|\sum_{n\geq
  n_0}a_nS^{n}|=r^{n_0}$, if $a_{n_0}\ne 0$, where  $|S|=r<1$ is a
nonzero real number\footnote{In the language of Berkovich, this field can be
  naturally identified with the complete
  residual field of the point $x_{0,r}\in \Ak$ (cf. section
  \ref{sec:berkovich-line-2}).}. We consider a rank one \emph{irregular} differential
module over $A=\CC\((S\))$ defined by the operator $\d+g:A\to A$,
where $g\in A$ has $S$-adic valuation $n_0$ that is less than or equal to $-2$.
We consider the connection $\nabla=\d+g$ as an element of
the $\CC$-Banach algebra $E=\LL{\CC}{\CC\((S\))}$ of bounded $\CC$-linear
maps of $\CC\((S\))$ with respect to the usual operator norm:
$\nor{\phi}=\Sup_{f\in A\setminus\{0\}}\frac{\nor{\phi(f)}}{\nor{f}}$, for all
$\varphi\in E$.
Then, the classical spectrum of the differentiel operator $\d+g$ is empty. Indeed,
since $\nor{\d}=\fra{r}$ and
$|(g-a)\-1|\leq r^2$ for all $a\in \CC$, then $\nor{(g-a)^{-1}\circ
  (\d)}\leq r$ (resp. $\nor{(\d)\circ(g-a)\-1}\leq r$) and the series
$\sum_{n\geq 0}(-1)^n\cdot((g-a)^{-1}\circ(\d))^n$ (resp. $\sum_{n\geq
  0}(-1)^n((\d)\circ (g-a)\-1)^n$) converges in
$\Lk{\CC\((S\))}$. Hence, for all $a\in \CC$, $(\sum_{n\geq
  0}(-1)^n\cdot((g-a)^{-1}\circ(\d))^n)\circ (g-a)^{-1}$
(resp. $(g-a)\-1\circ(\sum_{n\geq 0}(-1)^n\cdot((\d)\circ(g-a)\-1)^n$) is a left
(resp. right) inverse of $\d+(g-a)$
in $\Lk{\CC\((S\))}$ and $a$
does not belong to the spectrum.\footnote{Notice that it is
  relatively easy to show that any non trivial rank one  connection
  over $\CC\((S\))$ is set theoretically bijective. This follows
  from the classical index theorem of B. Malgrange \cite{Malgrange-Irreg}. However,
  the set theoretical inverse of the connection may not be
  automatically bounded. This is due to the fact that the base field
  $\CC$ is trivially valued and the Banach open mapping theorem does
  not hold in general. However, it is possible to prove that any such
  set theoretical inverse is bounded (cf. \cite{for}).}  
\end{exem}
To deal with this issue V. Berkovich understood that it was better not to define the spectrum as a subset of the base field $k$, but as a subset of the analytic line $\Ak$, which is a bigger space\footnote{This can be motivated by the fact that the resolvent is
  an analytic function on the complement of the spectrum.}. His theory of analytic
spaces (cf. \cite{Ber}, \cite{ber2}) enjoys several good local topological properties such as compactness, arc connectedness ...
In this setting
Berkovich developed a spectral theory for ultrametric operators  \cite[Chapter
7]{Ber}. The definition of the spectrum given by Berkovich is the
following: Let $(k,|.|)$ be complete field with respect to an
ultrametric absolute value and let
 $\Ak$ be the  Berkovich affine line. For a point $x\in\Ak$ we denote
 by $\Hx$ the associated complete residual field. We fix on $\Ak$  a coordinate function $T$. The spectrum $\Sigma_f(E)$ of an element $f$ of a $k$-Banach algebra $E$ is  the set
of points $x\in\Ak$, such that $f\ot 1-1\ot T(x)$ is not invertible in
$E\ct_k \Hx$. It can be proved (cf. \cite[Proposition 7.1.6]{Ber}) that this is equivalent to say that
there exists a complete valued field  $\Omega$ containing
isometrically $k$ and a constant $c\in \Omega$ such that
\begin{itemize}
\item the image of $c$ by the canonical projection $\A{\Omega}\to \Ak$
  is $x$;
  \item $f\ot 1-1\ot c$ is not invertible as an element of $E\ct_k\Omega$. 
  \end{itemize}
In some sense $\Omega=\Hx$ and $c=T(x)$ are the minimal possible
choices. This spectrum is compact, non empty and it satisfies the properties listed above
(cf. \cite[Theorem 7.12]{Ber}). Notice that if $f=\nabla$ is a
connection and if $E=\Lk{M}$ is the $k$-Banach algebra of continuous $k$-linear endomorphisms of
$M$,
then $f\ot 1-1\ot c$ is no more a connection as an element of
$E\ct_k\Omega$. Indeed, $\Lk{M}\ct_k\Omega$ does not coincide with
$\LL{\Omega}{M\ct_k \Omega}$, unless $\Omega$ is a finite
extension. In particular, {\em no index theorem can be used} to test the
set theoretical bijectivity of $\nabla\ot 1-1\ot c$. 

Coming back to the above example, using this definition it can be
proved that the spectrum of $\d+g$ is now reduced
 to the individual \emph{non-rational} point
 $x_{0,r^{n_0}}\in\Ac$(cf. \cite{for}). The dependence on $r$ shows that the Berkovich spectrum \emph{depends on the chosen absolute value on $k$}; whereas, instead, the classical spectrum is a completely algebraic notion.

Now, let $(k,|.|)$ be an ultrametric complete field. Let $X$ be an affinoid domain of $\Ak$ and $x$ a point of type
(2), (3) or (4). To avoid confusion we fix another coordinate function $S$ over
$X$. We set $A=\cO(X)$ or $\Hx$ and $d=g(S)\d$ with $g(S)\in A$. We will
distinguish the notion of differential equation in the sense of a
differential polynomial $P(d)$ acting on $A$ from the notion of
differential module $(M,\nabla)$ over $(A,d)$ associated to $P(d)$ in
a cyclic basis
(cf. section \ref{sec:spectr-assoc-diff-4}). Indeed, it is not hard to
prove 
that
the spectrum of $ P(\d)=(\d)^n+a_{n-1}(\d)^{n-1} +\cdots+a_0$ as an
element of $\Lk{A}$, where
$a_i\in k$, is given by the easy formula (cf. \cite[p. 2]{bousp} and Lemma~\ref{sec:berk-spectr-theory})
\begin{equation}
  \Sigma_{P(\d)}(\Lk{A})=P(\Sigma_{\d}(\Lk{A})).
\end{equation}
This set is always either a closed
disk or the topological closure of an open disk
(cf. Lemma~\ref{sec:berk-spectr-theory},
Remark~\ref{sec:spectr-line-diff}).
On the other hand, surprisingly enough, this differs form the
spectrum $\Sigma_\nabla(\Lk{M})$ of the differential module $(M,\nabla)$ associated to $P(\d)$ in a cyclic basis (i.e the
spectrum of $\nabla$ as an element of $\Lk{M}$). Indeed, this last is
a finite union of either closed disks or topological closures of open
disks
(cf. Theorem~\ref{sec:introduction}) centered on the roots of the (commutative)
polynomial $Q=X^n+a_{n-1}X^{n-1}+\cdots+a_0\in k[X]$ associated to
$P$.\footnote{Notice that $Q$ is the Fourier transform of $P$ and that its
  roots are related to the eigenvalues of the matrix $G$ of the
  connection in the companion form (cf. Remark \ref{sec:spectr-line-diff-1}).}
In order to introduce our next result, we denote by $\tilde{k}$ the
residual field of $k$ (cf. Section \ref{sec:defin-notat}) and we set:
\begin{equation}
  \label{eq:7}
  \omega= \begin{cases}
	|p|^{\frac{1}{p-1}} &\text{if char}(\tilde{k})=p \\ 1 &
          \text{if char}(\tilde{k})=0
	\end{cases}.
\end{equation}
The main statement of this paper
is then the following: 
\begin{Theo}\label{sec:introduction}
  We suppose that $k$ is algebraically closed. Let $X$ be a connected affinoid domain of $\Ak$ and $x\in\Ak$ a
  point of type (2), (3) or (4). We set $A=\cO(X)$ or $\Hx$. Let
  $(M,\nabla)$ be a differential module over $(A,\d)$ such that there
  exists a basis for which the associated matrix $G$ has constant
  entries (i.e. $G\in\cM_n(k)$). Let $\{a_1,\dots, a_N\}\subset k$  be the set of  eigenvalues of
$G$. Then the behaviour of the spectrum $\Sigma_\nabla(\Lk{M})$ of $\nabla$ as an element of
$\Lk{M}$ is summarized  in the following table:

\begin{center}
\resizebox{16cm}{!}{
  \begin{tabular}{|c|c|c|c|c|}
    \hline
    $A=$ & \multicolumn{2}{c|}{ $\cO(X)$} &\multicolumn{2}{c|}{
                                            $\Hx$}\\
    \cline{2-5}
              &$X=\disf{c}{r}$ & $X=\disf{c_0}{r_0}\setminus
                  \Bigcup_{i=1}^\mu\diso{c_i}{r_i}$& $x$ of type (2) or
                                                     (3)& $x$ of type (4)\\
    \hline
    & & & \multicolumn{2}{c|}{ }\\
    $\crk=p>0$ & $\Bigcup_{i=1}^N\disf{a_i}{\frac{\omega}{r}}$ &
                                                $\Bigcup_{i=1}^N\disf{a_i}{\frac{\omega}{\min\limits_{j}
                                                r_j}}$ &
                                                         \multicolumn{2}{c|}{$\Bigcup_{i=1}^N\disf{a_i}{\frac{\omega}{r(x)}}$ } \\
    \hline
    & & & & \\
    $\crk=0$ & $\Bigcup_{i=1}^N\overline{\diso{a_i}{\frac{1}{r}}}$& $\Bigcup_{i=1}^N\disf{a_i}{\frac{1}{\min\limits_{j}
                                                r_j}}$&
                                                        $\Bigcup_{i=1}^N\disf{a_i}{\frac{1}{r(x)}}$
                               &
                                 $\Bigcup_{i=1}^N\overline{\diso{a_i}{\frac{1}{r(x)}}}$\\
    \hline
  \end{tabular}
}
\end{center}

\end{Theo}

Using this result we can prove that when $x$ varies over Berkovich a
segment $(x_1,x_2)\subset \Ak$, the behaviour of the spectrum is left continuous, and continuous at
the points of type (3)
(cf. Section \ref{sec:spectrum-variation}).

The paper is organized  as follows. Section \ref{sec:defin-notat}
 is devoted to recalling the definitions and properties. It is divided
into three parts: in the first one we recall some basic definitions and
properties of $k$-Banach spaces, in the second we provide settings and notations related
to the affine line $\Ak$ and in the last one
 we recall the definition of the spectrum given by Berkovich and some
properties.

In Section \ref{sec:diff-modul-spectr}, we introduce the spectrum
associated to a differential module, and show how it behaves under
exact sequences. The main result of this section is the following:
\begin{pro}
		We suppose that $k$ is algebraically closed. Let $A$
                be as in Theorem~\ref{sec:introduction}, and let $d=g(S) d/dS$, with
                $g\in A$. Let
                $(M,\nabla)$ be a differential module over $(A,d)$
                such that there
  exists a basis for which the associated matrix $G$ has constant
  entries (i.e $G\in \cM_n(k)$). Then the spectrum of $\nabla$ is
                $\Sigma_{\nabla}=\bigcup\limits_{i=1}^{n}(a_i+\Sigma_d(\Lk{A}))$, where
                $\{a_1,\dots, a_n\}\subset k$ is the multiset of the eigenvalues of $G$.
               
              \end{pro}%
In particular the spectrum \emph{highly depends on the choice of the
  derivation $d$}.
In this paper we choose $d=\d$. 
This claim shows the importance of computing the spectrum of $\Sigma_d(\Lk{A})$. Therefore, in Section \ref{sec:stat-main-result}, a large part is devoted  to the
computation of the spectrum of
$\d$ acting on various rings.
In the last part of this section, we state and prove the main
result.

In the last section, we will explain, in the case of a
differential equation with constant coefficients, that the spectrum
associated to $(M,\nabla)$ over $(\Hx,d)$ satisfies a continuity
property, when $x$ varies over a segment $(x_1,x_2)\subset \Ak$.
\\%

{\bf Acknowledgments} The author wish to express her gratitude to her
advisors Andrea Pulita and Jérôme Poineau for their precious advice
and suggestions, and for carful reading. We also thank F. Beukers and F. Truc for
useful occasional discussions. 


\section{Definitions and notations}\label{sec:defin-notat}

All rings are with unit element. We denote by $\RR$ the field of real
numbers, and ${\R+=\{r\in \RR|r\geq
0\}}$. In all the paper $(k,|.|)$ will be a valued field of
characteristic 0, complete with
respect to an ultrametric absolute value $|.|:k\to \R+$
(i.e. verifying $|1|=1$, $|a\cdot b|=|a||b|$, $|a+b|\leq
\max(|a|,|b|)$ for all $a$, $b\in k$ and $|a|=0$ if and only if $a=0$). We set $|k|:=\{r\in \R+|\exists a\in k \text{
  such that } r=|a|\}$, $k^{\circ}:=\{a\in k||a|\leq 1\}$, ${k^{\circ\circ}:=\{a\in
k||a|< 1\}}$ and ${\tilde{k}:=k^\circ / k^{\circ\circ}}$. Let $E(k)$ be
the category whose objects are pairs $(\Omega,|.|)$, where $\Omega$ is a
field extension of $k$ complete  with respect to $|.|$, and whose
morphisms are the isometric rings morphisms. For $(\Omega,|.|)\in
E(k)$, we set $\Omega^{alg}$ to be an algebraic closure of $\Omega$,
the absolute value extends uniquely to an absolute value defined on
$\Omega^{alg}$. We denote by $\widehat{\Omega^{alg}}$ the
completion of $\Omega^{alg}$ with respect to this absolute value.  

\subsection{Banach spaces}\label{sec:banach-space-1}
An {\em ultrametric norm} on a $k$-vector space $M$ is a map $\nor{.}: M\to \R+$
verifying:
\begin{itemize}
\item $\nor{m}=0\Leftrightarrow m=0$.
\item $\forall m\in M$, $\forall \lambda\in k$; $\nor{\lambda
    m}=|\lambda|\nor{m}$.
\item $\forall m,\, n\in M$, $\nor{m+n} \leq \max(\nor{m}, \nor{n})$.
\end{itemize}

A normed $k$-vector space $M$ is a $k$-vector space endowed with an
ultrametric norm $\nor{.}$. If moreover $M$ is complete with respect
to this norm, we say that $M$ is {\em $k$-Banach space}. A $k$-linear map
$\phi: M\to N$ between two normed $k$-vector spaces is a {\em bounded
$k$-linear map} satisfying the following condition:
\[\exists C\in \R+,\; \forall m\in M; \q \nor{\phi(m)}\leq C\nor{m}.\]
If $C=1$ we say that $\phi$ is a {\em contracting map}.

Let $M$ and $N$ be two ultrametric normed
$k$-vector spaces. We endow the tensor product $M\ot_kN$ with the following
norm:
\begin{equation}
\label{eq:4}
\Fonction{\nor{.}: M\ot_kN}{ \R+}{f}{\inf \{
\max\limits_i\nor{m_i}\nor{n_i}|\q f=\sum\limits_i m_i\ot n_i\}}.
\end{equation}
The completion of
$M\ot_k N$ for this norm will be denoted by $M\ct_k N$.

An ultrametric norm on a $k$-algebra $A$ is an ultrametric norm
$\nor{.}: A\to \R+$ on the $k$-vector space $A$, satisfying the
additional properties:
\begin{itemize}
\item $\nor{1}=1$ if $A$ has a unit element.
\item $\forall m,\;n\in A;\q \nor{mn}\leq\nor{m}\nor{n}$.
\end{itemize}

A normed $k$-algebra $A$ is a $k$-algebra endowed with an
ultrametric norm $\nor{.}$. If moreover $A$ is complete with respect
to this norm, we say that $A$ is a {\em $k$-Banach algebra}.

We set ${\bf Ban}_k$ the category whose objects are  $k$-Banach vector
spaces and arrows are bounded $k$-linear maps. An isomorphism in this
category will be called bi-bounded isomorphism.  We set {\bf BanAL}$_k$
the category whose objects are $k$-Banach algebras and arrows are
bounded morphisms of $k$-algebras.

\begin{defi}\label{sec:banach-space}
  Let $A$ be a $k$-Banach algebra. The {\em spectral semi-norm} associated to the
  norm of $A$ is the map:
  \begin{equation}
    \label{eq:5}
    \Fonction{\Nsp{A}{.}: A} {\R+} {f} {\lim\limits_{n\to
    +\infty }{\nor{f^n}^{\frac{1}{n}}}=\inf\limits_{n\in \NN}\nor{f^n}^{\fra{n}}}.
\end{equation}
The existence of the limit and its equality with the infimum are well
known (cf. \cite[Définition 1.3.2]{bousp}).  If $\Nsp{A}{.}=\nor{.}$,
we say that $A$ is a {\em uniform algebra}.
  \end{defi}

  \begin{lem}\label{sec:banach-spaces}
    Let $M$ be a $k$-vector space and let $\nor{.}: M\to \R+$ be a map
    verifying the following properties: 
\begin{itemize}
\item $\nor{m}=0\Leftrightarrow m=0$.
\item $\forall m\in M$, $\forall \lambda\in k$; $\nor{\lambda
    m}\leq|\lambda|\nor{m}$.
\item $\forall m,\, n\in M$, $\nor{m+n} \leq \max(\nor{m}, \nor{n})$.
\end{itemize}
Then $\nor{.}$ is an ultrametric norm on $M$.
\end{lem}

\begin{proof}
  See \cite[Section 2.1.1, Proposition 4]{Bosc}. 
\end{proof}
  
  \begin{lem}\label{1}
		Let $\Omega\in E(k)$ and let $M$ be a $k$-Banach space. Then, the inclusion ${M\hookrightarrow M\ct_k \Omega}$ is an isometry. In particular, for all $v\in M$ and $c\in \Omega$ we have ${\nor{ v\ot c}=|c|\nor{v}}$.  
              \end{lem}
              
\begin{proof}
  Since $\Omega$ contains isometrically $k$, the morphism $M\ct_k
  \Omega\to \Omega$ (resp. $\Omega\to M\ct_k\Omega$) is an isometry
  (cf. \autocite[Lemma~3.1]{poi}). We know that $\nor{v\ot c}\leq |c|\nor{v}=|c|\nor{v\ot 1}$ for all $v\in
  M$ and $c\in\Omega$.  Therefore, by Lemma \ref{sec:banach-spaces}
  the tensor norm is a norm on $M\ct_k \Omega$ as an $\Omega$-vector
  space. Consequently, we obtain $\nor{v\ot c}=|c|\nor{v}$ for all
  $v\in M$ and $c\in\Omega$.
\end{proof}

\begin{pro}\label{2}
		Let $M$ be a $k$-Banach space, and $B$ be a uniform
                $k$-Banach algebra. Then in $M\ct_k B$ we have:
		\[ \forall m\in M, \; \forall f\in B, \; \nor{m\ot f}=\nor{m}\nor{f}. \]
		
	\end{pro}	
	
	\begin{proof}
		Let $\cM(B)$ be the analytic spectrum of $B$ (see
                \cite[Chapter 1]{Ber}). For
                all $x$ in $\cM(B)$ the canonical map $B\to
                \Hx$ is a contracting map, therefore the map
                $M\ct_kB\rightarrow M\ct_k\Hx$ is a contracting map
                too. Then, by Lemme~\ref{2} we have:
		
		\[\forall x \in \cM(B): \forall m\in M, \; \forall
                  f\in B, \; \nor{m\ot f(x)}=\nor{m}|f(x)|\leq\nor{m\ot f} \leq\nor{m}\nor{f} \] 
		thus,
		
		\[\forall m\in M, \; \forall f\in B, \; \nor{m}\max_{x \in \cM(B)}|f(x)|\leq\nor{m\ot f}\leq\nor{m}\nor{f}. \] 
		Since $B$ is uniform, we have $\max\limits_{x \in
                  \cM(B)}|f(x)|=\nsp{f}=\nor{f}$ (cf. \cite[Theorem 1.3.1]{Ber}), and $\nor{m}\nor{f}\leq\nor{m\ot f}\leq\nor{m}\nor{f}.$
		
              \end{proof}

Let $M$ and $N$ be two $k$-Banach spaces. We denote by $\Lk{M,N}$ the $k$-Banach
algebra of bounded $k$-linear maps $M\to N$ endowed with the operator norm:

\begin{equation}
  \label{eq:14}
  \Fonction{\Lk{M,N}}{\R+}{\phi}{\Sup_{m\in M\setminus\{0\}}\frac{\nor{\phi(m)}}{\nor{m}}}.
\end{equation}
We set $\Lk{M}:=\Lk{M,M}$. 

	\begin{lem}\label{3}
		Let  $\Omega\in E(k)$. There exists an isometric
                $k$-linear map ${\Lk{M}\ito \Lk{M\ct_k\Omega}}$ which
                extends to an $\Omega$-linear contracting map  $\Lk{M}\ct_k\Omega\to\LL{\Omega}{M\ct_k\Omega}$. 	
	\end{lem}
	
	\begin{proof}
		Let $\varphi\in\Lk{M}$, we have the bilinear map:
		
		\[
                    \Fonction{\varphi\times 1: M\times\Omega}{M\ot_k\Omega}
			 {(x,a)}{\varphi(x)\ot a} \]
		where $M\times \Omega$ is endowed with the product
                topology and $M\ot\Omega$ with the topology induced
                by tensor norm (cf. \eqref{eq:4}). The map $\varphi\times 1$ is continuous see \autocite[Chap. IV Lemma 17.1]{sch}. The diagram below:
		
		\[\xymatrix{M\times\Omega\ar[d]_{\ot}\ar[r]^{\varphi\times 1}& M\ot_k \Omega\\
			M\ot_k \Omega\ar[ur]_{\exists !\varphi\ot 1}}\] 
		is commutative, in addition $\varphi\ot 1$ is
                continuous see \autocite[Chap. IV Lemma 17.1]{sch}.
                By the universal property of the completion of a
                metric space, the map $\varphi\ot 1$
                extends to a continuous map $\varphi \ct 1: M\ct_k\Omega \to M\ct_k\Omega$. 
		So we obtain a $k$-linear map
                ${\Lk{M}\hookrightarrow \LL{\Omega}{M\ct_k\Omega}}$.
                We prove now that it is an isometry. Indeed, let $m\in M$ and $a\in\Omega$ then by  Lemma \ref{1}:
		\[\nor{\varphi\ot 1(m\ot a)}=\nor{\phi(m)\ot
                    a}=\nor{\varphi(m)}|a|\leq \nor{\varphi} \nor{m}
                  |a|=\nor{\phi}\nor{m\ot a}.\]
		Now, let $x=\sum\limits_i m_i\ot a_i\in M\ct_k\Omega$, then  we have:
		
		\[ \nor{\varphi\ct1(x)}\leq
                  \inf\{\max_i(\nor{\varphi(m_i)}|a_i|)\; |\;x=\sum_im_i\ot
                  a_i\}\leq \nor{\varphi}\inf\{\max_i \nor{m_i\ot a_i}
                  |\; x=\sum_im_i\ot
                  a_i \}. \]
		Consequently, \[\|\varphi\ct 1\|\leq \|\varphi\|. \]
		On other hand, since $ M\hookrightarrow M\ct_k\Omega$ is an isometry and $\varphi\ct 1_{|_M}=\varphi$, we have $\|\varphi\|\leq \|\varphi\ct 1\|.$\\
		The map $\Lk{M}\hookrightarrow
                \LL{\Omega}{M\ct_k\Omega}$ extends to a continuous
                $\Omega$-linear map
                ${\Lk{M}\ct_k\Omega\hookrightarrow
                  \LL{\Omega}{M\ct_k\Omega}}$. We need to prove now
                that it is a contracting map. Let ${\psi=\sum_i
                \varphi_i\ot a_i}$ be an element of $\Lk{M}\ct_k\Omega$, its image in
                $\LL{\Omega}{M\ct_k\Omega}$ is the element
                $\sum_ia_i\varphi_i\ct 1$. we have:

                \[\nor{\sum_i a_i\varphi_i\ct 1}\leq
                  \max_i{\nor{a_i\varphi_i\ct
                      1}}=\max_i\nor{\phi_i}|a_i|.\]
                Consequently,
                \[\nor{\sum_i a_i\varphi_i\ct 1}\leq
                 \inf\{\max_i\nor{\varphi_i}|a_i||\psi=\sum_i\varphi_i\ot
                 a_i\}=\nor{\psi}.\]
               Hence we obtain the result.

             \end{proof}
             
	\begin{lem}
		Let $M$ be a $k$-Banach space and $L$ be a finite
                extension of $k$. Then we have a bi-bounded isomorphism: \[\Lk{M}\ct_k L\simeq \LL{L}{M\ct_k L}.  \] 
	\end{lem}
	
	\begin{proof}
		As $L$ is a sequence of finite intermediary extensions generated by
                one element, by induction we can assume that
                $L=k(\alpha)$, where $\alpha$ is an algebraic element
                over $k$. By Lemma \ref{3} we have a morphism
                $\Lk{M}\ct_k L\to\LL{L}{M\ct_k L}$. As $L=k(\alpha)$,  we have a
                $k$-isomorphism $M\ct_k L\simeq
                \bigoplus_{i=0}^{n-1}M\ot (\alpha^i\cdot k)$. Let $\psi\in
                \LL{L}{M\ct L}$. The restriction ${\psi_{|_{M\ot
                    1}}:M\ot 1 \to M\ct_k L}$ is of the form ${m\ot
                1\mapsto \sum_{i=0}^{n-1}\varphi_i(m) \ot \alpha^i}$,
                where ${\varphi_i\in \cL_k(M)}$. As $\psi$ is
                $L$-linear, it is determined by the
                $\varphi_i$. This gives rise to an inverse $L$-linear
                map ${ \LL{L}{M\ct_k L}\to\Lk{M}\ct_kL} $. In the case
                where $k$ is not trivially valued, by the open mapping
                theorem (see \autocite[Section 2.8 ~Theorem of Banach]{Bosc}) the last
                map is bounded. Otherwise, the extension $L$ is
                trivially valued. Consequently, we have an isometric
                $k$-isomorphisms: ${L\simeq \bigoplus_{i=0}^{n-1}k}$
                equipped with the max norm. Therefore, we have ${\Lk{M}\ct_k L\simeq
                \bigoplus_{i=0}^{n-1}\Lk{M}}$ and ${M\ct_k L\simeq
                \bigoplus_{i=0}^{n-1}M}$ with respect to the max
                norm. Then we have:
                \[ \max_i\nor{\varphi_i}\leq \nor{\psi_{|_{M\ot 1}}} \leq
                  \nor{\psi}.\]
                Hence, we obtain the result.
              
              \end{proof}

              \subsection{Berkovich line}\label{sec:berkovich-line-2}
              In this paper, we will consider
$k$-analytic spaces in the sense of Berkovich (see \cite{Ber}). We denote by $\Ak$ the affine analytic line over the
ground field $k$, with coordinate $T$. We set $k[T]$ to be the ring of
polynomial with coefficients in $k$ and $k(T)$ its fractions
field.\\%

Recall that a point $x\in \Ak$ corresponds to a multiplicative
semi-norm $|.|_x$ on $k[T]$ (i.e. $|0|_x=0$, $|P+Q|_x\leq
\max(|P|_x,|Q|_x)$ and $|P\cdot Q|_x=|P|_x|Q|_x$ for all $P$, $Q\in k[T]$), that its restriction
coincides with the absolute value of $k$. The set $\mfp_x:=\{P\in
k[T]|\; |f|_x=0\}$ is a prime ideal of $k[T]$. Therefore, the
semi-norm extends to a multiplicative norm on the fraction field
Frac($A/\mfp_x)$).
\begin{nota}
  We denote by $\Hx$ the completion of Frac($A/\mfp_x$) with respect
  to $|.|_x$, and by $|.|$ the absolute value on $\Hx$ induced by
  $|.|_x$.
\end{nota}

Let $\Omega\in E(k)$ and $c\in
\Omega$. For $r\in \R+^*$ we set

\[\discf{\Omega}{c}{r}=\{x\in \A{\Omega}| |T(x)-c|\leq r\}\]
and 

\[\disco{\Omega}{c}{r}=\{x\in \A{\Omega}| |T(x)-c|< r\}\]
Denote by $x_{c,r}$ the unique point in the Shilov boundary of
$\discf{\Omega}{c}{r}$.\\
For $r_1$, $r_2\in \R+$, such that $0<r_1\leq r_2$ we set

\[\Couf{\Omega}{c}{r_1}{r_2}=\{x\in \A{\Omega}| r_1\leq |T(x)-c| \leq
  r_2\}\]
and for $r_1<r_2$ we set:

\[\Couo{\Omega}{c}{r_1}{r_2}=\{x\in \A{\Omega}| r_1< |T(x)-c|<r_2\}\]
We may delete the index $\Omega$ when it is obvious from the context. 

\begin{hyp}
  Until the end of this section we will suppose that $k$ is
  algebraically closed.
\end{hyp}

Each  affinoid domain $X$ of
$\Ak$ is a finite  union of connected affinoid domain of the form:

\begin{equation}
  \label{eq:2}
   \disf{c_0}{r_0}\setminus\bigcup_{i=1}^\mu\diso{c_i}{r_i}
\end{equation}
Where $c_0,\dots, c_\mu\in \disf{c_0}{r_0}\cap k$ and $0<r_1,\dots, r_\mu\leq
r_0$ (the case where $\mu=0$ is included).

Let $X$ be an affinoid domain of $\Ak$, we denote by $\cO(X)$ the
$k$-Banach algebra of global sections of $X$. For a disk $\disf{c}{r}$
we have:
\[\fdisf{c}{r}=\{\sum_{i\in \NN} a_i(T-c)^i|\; a_i\in k,\; \lim\limits_{i\to+\infty}|a_i|r^i= 0\}\]
equipped with the multiplicative norm:
\[\nor{\sum_{i\in \NN}a_i(T-c)^i}=\max_{n\in \NN}|a_i|r^i\]
More generally, if $X$ is of the form \eqref{eq:2}, then by the Mittag-Leffler
decomposition \cite[Proposition~2.2.6]{Van}, we have:
\[\cO(X)=\bigoplus\limits_{i=1}^\mu\{\sum_{j\in \NN^*}
  \dfrac{a_{ij}}{(T-c_i)^j}|\; a_{ij}\in k,\; \lim\limits_{j\to+\infty}|a_{ij}|r_i^{-j}= 0 \}\oplus\fdisf{c_0}{r_0}.\]
where
$\nor{\sum_{j\in\NN^*}\frac{a_{ij}}{(T-c_i)^j}}=\max_j|a_{ij}|r_i^{-j}$ and
the sum above is equipped with the maximum norm.

For $\Omega\in E(k)$, we
set $X_\Omega=X\ct_k \Omega$. We have a canonical projection of analytic
spaces:
\begin{equation}
  \label{eq:6}
  \pi_{\Omega/k}: X_\Omega\to X
\end{equation}

\begin{defi}\label{sec:berkovich-line-3}
  Let $x\in\Ak$. we define the radius of $x$ to be the value:
  \[r_k(x)=\inf\limits_{a\in k}|T(x)-a|.\]
  We may delete $k$ if it is obvious from the context.
\end{defi}

We can describe the
field $\Hx$, where $x$ is a point of  $\Ak$, in a more
explicit way. In the case where $x$ is of type (1), we have
$\Hx=k$. If $x$ is of type (3) of the form $x=x_{c,r}$ where
$c\in k$ and $r\notin |k|$, then it is easy to see that $\Hx=\fcouf{c}{r}{r}$. But for the points
of type (2) and (4), a description is not obvious, we have the following
Propositions:

\begin{pro}[Mittag-Leffler Decomposition]\label{15}
  Let
                $x=x_{c,r}$ be a point of type (2) of $\Ak$ ($c\in k$
                and $r\in |k^*|$). We have the decomposition: \[\Hx=E\oplus \fdisf{c}{r}  \]
		where $E$ is the closure  in $\Hx$ of the ring of rational fractions of $k(T-c)$ whose poles are in $D^+(c,r)$.
		i.e. for $\gamma\in k$ with $|\gamma|=r$:
		\[E:=\widehat{\bigoplus}_{\alpha\in
                    \tilde{k}}\{\sum_{i\in\NN^*} \frac{a_{\alpha
                      i}}{(T-c+\gamma\alpha)^i}|\; a_{\alpha i}\in
                  k,\; \lim\limits_{i\to +\infty}|a_{\alpha i}|r^{-i}= 0  \}.\] 
		
	\end{pro}
	
	\begin{proof}
		In the case where $k$ is not trivially valued we refer
                to \cite[Theorem~2.1.6]{Chr}. Otherwise, the only point
                of type (2) of $\Ak$ is $x_{0,1}$, which corresponds
                to the trivial norm on $k[T]$. Therefore, we have $\Hx=k(T)$. 
	\end{proof}

        \begin{lem}\label{sec:berkovich-line-1}
         
          Let $x\in\Ak$ be a
          point of type (4). The field $\Hx$ is the completion of
          $k[T]$ with respect to the norm $|.|_x$.
        \end{lem}
        \begin{proof}
Recall that for a point $x\in \Ak$ of type (4), the field $\Hx$ is the
completion of $k(T)$ with respect to $|.|_x$. To prove that $\Hx$ is
the completion of $k[T]$, it is enough to show that $k[T]$ is dense in
$k(T)$ with respect to $|.|_x$.
For all $a\in k$ it is then enough to show that  there exists a sequence $(P_i)_{i\in\NN}\subset k[T]$
which converges to $\frac{1}{T-a}$.
Let  $a\in k$. Since $x$ is of type (4),  there exists $c\in k$ such
that $|T-c|_x<|T-a|_x$. Therefore we have $|c-a|_x=|T-a|_x$ and we obtain:   

\[ \frac{1}{T-a}=\frac{1}{(T-c)+(c-a)}=\frac{1}{c-a}\sum_{i\in
    \NN} \frac{(T-c)^i}{(a-c)^i}.\]
So we conclude. 
        \end{proof}
        
        \begin{pro}[\cite{Chr}]\label{14}
          Let $x\in \Ak$
          be a point of type (4). Then there exists an isometric
          isomorphism $\psi:\fdiscf{\Hx}{T(x)}{r_k(x)}\to \Hx\ct_k\Hx$ of $\Hx$-Banach algebras.
        \end{pro}

        \begin{proof}
          Note that, for any element $f\in \Hx\ct_k\Hx$ with $|f|\leq
          r_k(x)$, we can define a morphism of $\Hx$-Banach algebra
           $\psi:\fdiscf{\Hx}{T(x)}{r_k(x)}\to \Hx\ct_k \Hx$, that
           associates $f$ to $T-T(x)$. To prove the statment we choose
          $f=T(x)\ot 1-1\ot T(x)$. Indeed, for all $a\in k$  we have
          $T(x)\ot 1-1\ot T(x)=(T(x)-a)\ot 1- 1\ot(T(x)-a)$. Hence,
          \[|T(x)\ot 1- 1\ot T(x)|\leq \inf_{a\in k}(\max(|(T(x)-a)|,|(T(x)-a)|)=\inf_{a\in k}|T(x)-a|=r_k(x)\] 
        By construction $\psi$ is a contracting $\Hx$-linear
        map. In order to prove that it is an isometric isomorphism, we need to
        construct its inverse map and show that it is also a
        contracting map. For all $a\in k$, $|T(x)-a|>r_k(x)$, hence
        $T-a$ is invertible in $\fdiscf{\Hx}{T(x)}{r_k(x)}$. This
        means that $k(T)\subset \fdiscf{\Hx}{T(x)}{r_k(x)}$ as
        $k$-vector space. As for all $a\in k$ we have:
        \[|T-a|=\max(r_k(x),|T(x)-a|)=|T(x)-a|\]
        the restriction of the norm of
        $\fdiscf{\Hx}{T(x)}{r_k(x)}$ to $k(T)$ coincides  with
        $|.|_x$. Consequently, the closure of $k(T)$ in
        $\fdiscf{\Hx}{T(x)}{r_k(x)}$ is exactly $\Hx$, which means that
        we have an isometric embedding $\Hx \hookrightarrow
        \fdiscf{\Hx}{T(x)}{r_k(x)}$ of $k$-algebras which associates
        $T(x)$ to $T$. This map extends uniquely to a contracting morphism of
        $\Hx$-algebras:
\[\xymatrix{\Hx\ar@{^{(}->}[r]\ar[d]& \fdiscf{\Hx}{T(x)}{r_k(x)}\\
    \Hx\ct_k\Hx\ar[ru]_{\exists! \phi}& \\}\]
        Then we have
        $\varphi(T(x)\ot 1-1\ot T(x))=T-
        T(x)$. Since $(T-T(x))$ (resp. $T(x)\ot 1-1\ot T(x)$) is a
        topological generator of the $k$-algebra
        $\discf{\Hx}{T(x)}{r(x)}$ (resp. $\Hx\ot\Hx$) and both of $\varphi\circ\psi$ and $\psi\circ\varphi$ are bounded morphisms of $k$-Banach algebras,
        we have ${\varphi\circ
        \psi=Id_{\fdiscf{\Hx}{T(x)}{r_k(x)}}}$ and ${\psi\circ
        \varphi=Id_{\Hx\ct_k\Hx}}$. Hence, we obtain the result.
    \end{proof}

    For any point $x$ of
    $\Ak$ and any extension $\Omega\in E(k)$ the tensor norm on the
    algebra $\Hx\ct_k \Omega$ is multiplicative (see \cite[Corollary
    3.14.]{poi}). We denote this norm by $\uni{\Omega}(x)$.
    \begin{pro}\label{sec:berkovich-line}
    
      Let $x\in\Ak$ be a
     point of type (i), where $i\in\{2,3, 4\}$. Let $\Omega\in E(k)$
     algebraically closed. If
     $\Omega\notin E(\Hx)$, then $\pik{\Omega}\-1\{x\}=\{\uni{\Omega}(x)\}$. 
   \end{pro}

   \begin{proof}
     Recall that if $\pik{\Omega}\-1\{x\}\setminus \{\uni{\Omega}(x)\}$
      is not empty, then it is a union of disjoint open disks
     (cf. \cite[Theorem~2.2.9]{np2}). Therefore, it contains points of
     type (1) which implies that $\Omega\in E(\Hx)$. Hence, we obtain
     a contradiction.
   \end{proof}

\subsection{Berkovich's spectral theory} We recall here the definition of the sheaf of analytic functions with value
in a $k$-Banach space over an analytic space and the definition
of the spectrum given by V. Berkovich in \cite{Ber}.
\begin{defi}
  Let $X$ be a $k$-affinoid space and $B$ be a $k$-Banach space. We
  define the sheaf of analytic functions with values in $B$ over $X$ to
  be the sheaf:
  \[\cO_X(B)(U)=\lim_{\substack{\longleftarrow\\V\subset U}} B\ct_k A_V\]
  where $U$ is an open subset of $X$, $V$ an affinoid domain and
  $A_V$ the $k$-affinoid algebra associated to $V$.
\end{defi}

As each $k$-analytic space is obtained by gluing $k$-affinoid spaces
(see \cite{Ber}, \cite{ber2}), we can extend
the definition to $k$-analytic spaces. Let $U$ be an open subset of
$X$. Every element $f\in\cO_X(B)(U)$ induces a fonction:
$f:\;U\to \coprod_{x\in U}B\ct_k\Hx$,  ${x\mapsto f(x)}$,
where $f(x)$
is the image of $f$ by the map $\cO_X(B)(U)\to B\ct_k\Hx$. We will call an
{\em analytic function over $U$ with value in $B$}, any function $\psi:
U\to \coprod_{x\in U}B\ct_k\Hx$ induced by an element $f\in\cO_X(B)(U)$.

\begin{hyp}
  Until the end of the paper, we will assume that all  Banach algebras are
  with unit element.
\end{hyp}
\begin{defi}
  Let $E$ be $k$-Banach algebra and $f\in E$.
  The spectrum of $f$ is the set $\Sigma_{f,k}(E)$ of points $x\in
  \Ak$ such that the element $f\ot 1-1\ot T(x)$ is not invertible in
  the $k$-Banach algebra $E\Ct{k}\Hx$.
  The resolvent of $f$ is the function:

  \[ \begin{array}[c]{rcl}
                R_f:\Ak\setminus \Sigma_{f,k}(E)&\longrightarrow
       &\coprod\limits_{{x\in\Ak\setminus\Sigma_{f,k}(A)} } E\Ct{k}\Hx\\
                x& \longmapsto& (f\ot1-1\ot T(x))^{-1}\\
              \end{array}\]

\end{defi}

\begin{rem}
  If there is no confusion we denote the spectrum of $f$, as an element of $E$,  just by $\Sigma_f$.
  \end{rem}

  \begin{Theo}\label{sec:defin-basic-propr}
	Let $E$ be a Banach $k$-algebra  and $f\in E$. Then:
	\begin{enumerate}
		\item  The spectrum $\Sigma_f$ is a non-empty compact subset of $\Ak$.
		\item The radius of the smallest (inclusion) closed disk with center at zero which contains $\Sigma_f$ is equal to $\nsp{f}$.
		\item The resolvent $R_f$ is an analytic function on $\Pk\setminus \Sigma_f$ which is equal to zero at infinity.	 
	\end{enumerate}
      \end{Theo}

      \begin{proof} See \cite[Theorem 7.1.2]{Ber}. \end{proof}

      \begin{rem}\label{sec:berk-spectr-theory-2}
        Let $E$ be a $k$-Banach algebra  and $r\in\R+^*$. Note that we have
        \[\cO_{\Ak}(E)(\disf{a}{r})=E\ct_k\fdisf{a}{r},\] i.e any
        element of $\cO_{\Ak}(E)(\disf{a}{r})$ has the form
        $\sum_{i\in \NN}f_i\ot(T-a)^i$ with $f_i\in E$. Let
        $\phi=\sum_{i\in\NN}f_i\ot (T-a)^i$ be an element of $\cO_{\Ak}(E)(\disf{a}{r})$. Since $\nor{f_i\ot
          (T-a)^i}=\nor{f_i}\nor{(T-a)}^i$ in
        $\cO_{\Ak}(E)(\disf{a}{r})$ (cf. Proposition~\ref{2}), the
        radius of convergence of $\phi$ with respect to $T-a$ is equal
        to $\liminf\limits_{i\to +\infty}\nor{f_i}^{-\frac{1}{i}}$.
      \end{rem}
      
      \begin{lem}\label{sec:defin-basic-propr-1}
       We maintain the same assumption as in
       Theorem~\ref{sec:defin-basic-propr}. If $a\in (\Ak\setminus \Sigma_f)\cap k$, then the biggest open
       disk centred in $a$ contained in $\Ak\setminus\Sigma_f$ has radius $R=\nsp{(f-a)\-1}\-1$.
     \end{lem}

     \begin{proof}
       Since $\Sigma_f$ is compact and not empty, the biggest disk
       $\diso{a}{R}\subset\Ak\setminus\Sigma_f$ has finite positive
       radius $R$. In the
       neighbourhood of the point $a$, we have:
       \[R_f=(f\ot 1-1\ot T)\-1=((f-a)\ot 1-1\ot
         (T-a))\-1=((f-a)\-1\ot 1)\sum_{i\in\NN}\frac{1}{(f-a)^i}\ot
         (T-a)^i.\]
       On the one hand the radius of convergence of the latter series  with
       respect to $(T-a)$ is equal to $\nsp{(f-a)\-1}\-1$
       (cf. Remark~\ref{sec:berk-spectr-theory-2}). On the other hand, the
       analyticity of $R_f$ on $\diso{a}{R}$ implies that the radius
       of convergence with respect to $T-a$ is equal to $R$. Hence, we
       have $R=\nsp{(f-a)\-1}\-1$.
     \end{proof}

     \begin{pro}\label{sec:berk-spectr-theory-4}
     Let $E$ be a commutative $k$-Banach algebra  element,
     and $f\in E$. The spectrum $\Sigma_f$ of  $f$ coincides with the
     image of the analytic spectrum $\cM(E)$ by the map induced by the
     ring morphism $ k[T]\to E$, $T\mapsto f$.
   \end{pro}
   \begin{proof}
     See \cite[Proposition~7.1.4]{Ber}.
   \end{proof}

   \begin{defi}
     Let $E$ be a $k$-Banach algebra and $B$ a commutative
     $k$-subalgebra of $E$. We say that $B$ is a maximal subalgebra of
     $E$, if for any subalgebra $B'$ of $E$ we have the following
     property:
     \[(B\subset B'\subset E)\Leftrightarrow(B'=B).\] 
   \end{defi}

   \begin{rem}
     A maximal subalgebra $B$ is necessarily closed in $E$, hence a
     $k$-Banach algebra.
   \end{rem}
   
   \begin{pro}\label{sec:berk-spectr-theory-5}
     Let $E$ be a $k$-Banach algebra. For any
     maximal commutative subalgebra $B$ of $E$, we have:
     \[\forall f\in B;\qqq \Sigma_f(B)=\Sigma_f(E).\]
   \end{pro}
   \begin{proof}
     See \cite[Proposition~7.2.4]{Ber}.
   \end{proof}
   
     Let $P(T)\in k[T]$, let $E$ be a Banach $k$-algebra and let $f\in
      E$. We set $P(f)$ to be the image of $P(T)$ by the morphism
      $k[T]\to E$, $T\mapsto f$, and $P: \Ak \to \Ak$ to be the
      analytic map associated to $k[T]\to k[T]$, $T\mapsto P(T)$. 

      \begin{lem}\label{sec:berk-spectr-theory}
        Let $P(T)\in k[T]$, let $E$ be a Banach $k$-algebra and let $f\in
      E$. We have the equality of sets:
        \[\Sigma_{P(f)}=P(\Sigma_f)\]
      \end{lem}

      \begin{proof}
   Let $B$ a maximal commutative $k$-subalgebra of $E$ containing $f$
   (which exists by Zorn's Lemma). Then $B$ contains also $P(f)$. By
   Proposition~\ref{sec:berk-spectr-theory-5} we have
   $\Sigma_f(E)=\Sigma_f(B)$ and $\Sigma_{P(f)}(E)=\Sigma_f(B)$. Let
   $^*f:\cM(B)\to \Ak$ (resp. $^*P(f):\cM(B)\to \Ak$) be the map
   induced by $k[T]\to E$, $T\mapsto f$ (resp. $T\mapsto P(f)$). By
   Proposition~\ref{sec:berk-spectr-theory-4} we have
   $\Sigma_f(B)=^*f(\cM(B))$ and
   $\Sigma_{P(f)}(B)=^*P(f)(\cM(B))$. Since $^*P(f)=P\circ ^*f$, we
   obtain the equality. 
 \end{proof}

 \begin{rem}
   Note that, we
   can imitate the proof provided in \cite[p.2]{bousp} to prove the
   statement of Lemma~\ref{sec:berk-spectr-theory}. 
 \end{rem}

      \begin{lem}\label{sec:berk-spectr-theory-1}
        Let $E$ and $E'$ be two Banach $k$-algebras and $\varphi: E\to
        E'$ be a bounded morphism of $k$-algebras. If $f\in E$ then we have:
        \[\Sigma_{\varphi(f)}(E')\subset \Sigma_f(E).\]
        If moreover $\varphi$ is a bi-bounded isomorphism then we have the
        equality.
      \end{lem}
      \begin{proof}
       Consequence of the definition.
      \end{proof}

Let $M_1$ and $M_2$ be two $k$-Banach spaces, let $M=M_1\oplus M_2$
endowed with the max norm (i.e.  $\forall m_1\in M_1$ and $\forall
m_2\in M_2$ $\nor{m_1+m_2}=\max(\nor{m_1},\nor{m_2})$ ). We set:
\[\cM(M_1,M_2)=\left\{\begin{pmatrix}L_1&L_2\\L_3&L_4\end{pmatrix}|L_1\in \Lk{M_1}, L_2\in \Lk{M_2,M_1}, L_3\in \Lk{M_1,M_2}, L_4\in \Lk{M_2}\right\}\]
We define the multiplication in $\cM(M_1,M_2)$ as follows: 

	\[\begin{pmatrix}A_1&A_2\\A_3&A_4\end{pmatrix}\begin{pmatrix}B_1&B_2\\B_3&B_4\end{pmatrix}= \begin{pmatrix}A_1B_1+A_2B_3&A_1B_2+A_2B_4\\A_3B_1+A_4B_4&A_3B_2+A_4B_4\end{pmatrix}. \]
        Then $\cM(M_1,M_2)$ endowed with the max norm is a $k$-Banach algebra.
        \begin{lem}\label{21}
          We have a bi-bounded isomorphism of $k$-Banach algebras:
          \[\Lk{M}\simeq \cM(M_1,M_2).\]

        \end{lem}

        \begin{proof}
          Let $p_j$ be the projection of $M$ onto $M_j$ and $i_j$ be
          the inclusion of $M_j$ into $M$, where $j\in \{1,2\}$. We
          define the following two $k$-linear maps:
          \[\Fonction{\Psi_1:
              \Lk{M}}{\cM(M_1,M_2)}{\phi}{\begin{pmatrix} p_1\varphi
                i_1 & p_1\varphi i_2 \\ p_2\varphi i_1 & p_2\varphi
                i_2 \end{pmatrix}}\]
          \[\Fonction{\Psi_2: \cM(M_1,M_2)}{\Lk{M}}{\begin{pmatrix}L_1&L_2\\L_3&L_4\end{pmatrix}}{i_1L_1p_1+i_1L_2p_2+i_2L_3p_1+i_2L_4p_2}\]
        Since the projections and inclusions are bounded maps, then
        the maps $\Psi_1$ and $\Psi_2$ are bounded too. It is easy to
        show that $\Psi_1\circ \Psi_2=id_{\cM(M_1,M_2)}$ and $\Psi_2\circ\Psi_1=id_{\Lk{M}}$. Hence we have an isomorphism of $k$-Banach spaces.                                
              
      \end{proof}
      
      We will need the following lemma for the computation of the
      spectrum:
      \begin{lem}\label{10}
	Let $M_1$ and $M_2$ be $k$-Banach spaces and let 
        $M=M_1\oplus M_2$ endowed with the max norm. Let $p_1$, $p_2$ be the respective
        projections associated to $M_1$ and $M_2$ and $i_1$, $i_2$ be
        the respective inclusions. Let $\varphi\in \Lk{M}$and set
        $\varphi_1=p_1\varphi i_1 \in \Lk{M_1}$ and
        $\varphi_2=p_2\varphi i_2\in \Lk{M_2}$. If
        $\varphi(M_1)\subset M_1$ , then we have:
        \begin{itemize}
        \item[i)] $\Sigma_{\phi_i}(\Lk{M_i})\subset\Sigma_\phi(\Lk{M})\cup \Sigma_{\phi_j}(\Lk{M_j})$,
          where $i,\,j\in \{1,2\}$ and $i\not=j$.
        \item[ii)] $\Sigma_\varphi(\Lk{M})\subset
          \Sigma_{\varphi_1}(\Lk{M_1})\cup\Sigma_{\varphi_2}(\Lk{M_2})$. Furthermore,
          if $\varphi(M_2)\subset M_2$, then we have the equality.
        \item[iii)] If
          $\Sigma_{\phi_1}(\Lk{M_1})\cap\Sigma_{\phi_2}(\Lk{M_2})=\emptyset$,
          then  $\Sigma_\varphi(\Lk{M})=
          \Sigma_{\varphi_1}(\Lk{M_1})\cup\Sigma_{\varphi_2}(\Lk{M_2})$.
        \end{itemize}
      \end{lem}
      
\begin{proof}
	By Lemma~\ref{21}, we can represent the elements of $\Lk{M}$ as follows:
	\[\Lk{M}=\left\{\begin{pmatrix}L_1&L_2\\L_3&L_4\end{pmatrix}|L_1\in \Lk{M_1}, L_2\in \Lk{M_2,M_1}, L_3\in \Lk{M_1,M_2}, L_4\in \Lk{M_2}\right\}\] 
	and $\varphi$ has the form $\begin{pmatrix}\varphi_1
          &L\\0&\varphi_2\end{pmatrix}$, where $L\in\Lk{M_2,M_1}$.\\ 
	Let $x\in\Ak$. We have an isomorphism of $k-$Banach algebras:		
	\[\Lk{M}\ct_k\Hx=\left\{\begin{pmatrix}L_1&L_2\\L_3&L_4\end{pmatrix}|\begin{array}{c}L_1\in
                                                                               \Lk{M_1}\ct_k\Hx,
                                                                               L_2\in
                                                                               \Lk{M_2,M_1}\ct_k\Hx,\\
                                                                               L_3\in
                                                                               \Lk{M_1,M_2}\ct_k\Hx,
                                                                               L_4\in
                                                                               \Lk{M_2}\ct_k\Hx\end{array}\right\}\]
	Consequently, 
	\[\varphi\ot1-1\ot T(x)=\begin{pmatrix}\varphi_1\ot1-1\ot T(x)&L\ot1\\0&\varphi_2\ot1-1\ot T(x)\end{pmatrix}.\]

We first prove i). Let $ \begin{pmatrix}
    L_1& C\\ 0& L_2\\
  \end{pmatrix}
$ be an invertible element of $\Lk{M}\ct_k\Hx$. We claim that if, for $i\in\{1,2\}$, $L_i$ is invertible
in $\Lk{M_i}\ct_k\Hx$, then so  is $L_j$, where
$j\not=i$. Indeed, let $\begin{pmatrix}
    L_1'& C'\\ B& L_2'\\
  \end{pmatrix}$ such that we have:

\[\begin{pmatrix}
    L_1& C\\ 0& L_2\\
  \end{pmatrix} \begin{pmatrix}
    L_1'& C'\\ B& L_2'\\
  \end{pmatrix}=\begin{pmatrix} 1&0\\ 0& 1\\ \end{pmatrix};\qq \begin{pmatrix}
    L_1'& C'\\ B& L_2'\\
  \end{pmatrix} \begin{pmatrix}
    L_1& C\\ 0& L_2\\
  \end{pmatrix}=\begin{pmatrix} 1&0\\ 0& 1\\ \end{pmatrix}.\]
Then we obtain:
\[
  \begin{cases}
    L_1L_1'+CB=1\\
L_1C'+CL_2=0\\
L_2B=0\\
L_2L_2'=1\\
  \end{cases}
;\qq
\begin{cases}
  L'_1L_1=1\\ 
L_1'C+C'L_2=0\\
BL_1=0\\
BC+L_2'L_2=1\\
\end{cases}.\]
We deduce that $L_1$ is left invertible and $L_2$ is right
invertible. If $L_1$ is invertible, then $B=0$ which implies that
$L_2$ is left invertible, hence invertible. If $L_2$ is invertible,
then $B=0$ which implies that $L_1$ is right invertible, hence invertible.
Therefore, if $\phi\ot 1-1\ot T(x)$ and
$\phi_i\ot 1-1\ot T(x)$ are invertible where $i\in\{1,2\}$, then
$\phi_j\ot 1-1\ot T(x)$ is invertible for $j\in
\{1,2\}\setminus\{i\}$. We conclude that
$\Sigma_{\phi_j}\subset\Sigma_\phi\cup\Sigma_{\phi_i}$ where $i,\,
j\in\{1,2\}$ and $i\not=j$.

We now prove ii). If $\varphi_1\ot1-1\ot T(x)$ and $\varphi_2\ot1-1\ot T(x)$ are
invertible, then $\varphi\ot1-1\ot T(x)$ is invertible. This proves that $\Sigma_{\varphi}\subset\Sigma_{\varphi_1}\cup \Sigma_{\varphi_2}$.
If $\varphi (M_2)\subset M_2$, then $L=0$ which implies that: if
$\varphi\ot1-1\ot T(x)$ is invertible, then $\varphi_1\ot1-1\ot T(x)$
and $\varphi_2\ot1-1\ot T(x)$ are invertible. Hence we have the
equality.

We now prove iii). If $\Sigma_{\phi_1}\cap\Sigma_{\phi_2}=\emptyset$,
then by above we have $\Sigma_{\phi_1}\subset\Sigma_{\phi}$ and
$\Sigma_{\phi_2}\subset\Sigma_{\phi}$. Therefore, $\Sigma_{\phi_1}\cup\Sigma_{\phi_2}\subset\Sigma_{\phi}$.

\end{proof}

\begin{rem}\label{16}
	Set notations as in Lemma \ref{10}. In
        the proof above, we showed also: if
        $\varphi\ot1-1\ot T(x)$ is invertible, then
        $\varphi_1\ot1-1\ot T(x)$ is left invertible and
        $\varphi_2\ot1-1\ot T(x)$ is right invertible.
\end{rem}

      \section{Differential modules and spectrum}\label{sec:diff-modul-spectr}

      \subsection{Preliminaries}\label{sec:preliminaries}

      Recall that a differential $k$-algebra, denoted by  $(A,d)$, is a commutative
      $k$-algebra $A$ endowed with a $k$-linear derivation $d:A\to
      A$. A differential module $(M,\nabla)$ over $(A,d)$
      is a finite free $A$-module $M$ equipped with a $k$-linear map
      $\nabla:M\to M$, called connection of $M$, satisfying
      $\nabla(fm)=df.m+f.\nabla(m)$ for all $f\in A$ and $m\in M$. If
      we fix a basis of $M$, then we get an isomorphism of
      $A$-modules $M\tilde{\to} A^n$, and the operator $\nabla$ is
      given in this basis by the rule:

      \begin{equation}
        \label{eq:1}
        \nabla \begin{pmatrix} f_1\\ \vdots \\ f_n \end{pmatrix}
		= \begin{pmatrix} df_1\\ \vdots \\ df_n \end{pmatrix}+G 
		\begin{pmatrix} f_1\\ \vdots \\ f_n \end{pmatrix}
      \end{equation}
    where $G\in \cM_n(A)$ is a matrix. Conversely the data of such
    a matrix defines a differential module structure on $A^n$ by
    the
    rule \eqref{eq:1}.

    A morphism between differential modules is a
    $k$-linear map $M\to N$ commuting with connections.
    We set $A\<D\>=\bigoplus\limits_{i\in \NN}A.D^i$ to be the ring of
    differential polynomials equipped with the non-commutative
    multiplication  defined by the rule: $D.f=df+f.D$ for all $f\in A$. Let
    $P(D)=g_0+\dots +g_{\nu-1}D^{\nu-1}+D^{\nu}$ be a monic
    differential polynomial. The quotient $A\<D\>/A\<D\>.P(D)$ is a
    finite free $A$-module of rank $\nu$. Equipped with the
    multiplication by $D$, it is a differential module over
    $(A,d)$. In the basis $\{1, D,\dots, D^{\nu-1}\}$ the
    multiplication by $D$ satisfies:
    \[
      D %
    \left(
     \raisebox{0.5\depth}{%
       \xymatrixcolsep{1ex}%
       \xymatrixrowsep{1ex}%
       \xymatrix{f_1\ar@{.}[ddddd]\\ \\\\ \\ \\f_{\nu}\\}%
       }
   \right)
     =
\left(
     \raisebox{0.5\depth}{%
       \xymatrixcolsep{1ex}%
       \xymatrixrowsep{1ex}%
       \xymatrix{df_1\ar@{.}[ddddd]\\ \\ \\ \\ \\d f_{\nu}\\}%
       }
   \right)
   +
      \left(
     \raisebox{0.5\depth}{%
       \xymatrixcolsep{1ex}%
       \xymatrixrowsep{1ex}%
       \xymatrix{0\ar@{.}[rrr]& & & 0&-g_0\ar@{.}[dddd] \\
        1\ar@{.}[rrrddd]& 0\ar@{.}[rr]\ar@{.}[rrdd]& &0\ar@{.}[dd]& \\
        0\ar@{.}[dd]\ar@{.}[rrdd]& &  &  &\\
        & &  &0& \\
        0\ar@{.}[rr]& &0&1& -g_{\nu-1}\\ }%
        }
       \right)
  \left(
     \raisebox{0.5\depth}{%
       \xymatrixcolsep{1ex}%
       \xymatrixrowsep{1ex}%
       \xymatrix{f_1\ar@{.}[ddddd]\\ \\ \\\\ \\f_{\nu}\\}%
       }
   \right)
  \]
  \begin{Theo}[The cyclic vector theorem]
    Let $(A,d)$ be a $k$-differential field (i.e $A$ is a field),
    with $d\not=0$, and let
    $(M,\nabla)$ be a differential module over $(A,d)$ of rank $n$. Then there
    exists $m\in M$ such that $\{m, \nabla(m),\dots \nabla^{n-1}(m)\}$
    is a basis of $M$. In this case we say that $m$ is cyclic vector.
  \end{Theo}

  \begin{proof}
    See \cite[Theorem ~5.7.2.]{Ked}.
  \end{proof}

  \begin{rem}
    The last theorem means that there exists an isomorphism of
    differential modules between $(M,\nabla)$ and $(A\<D\>/A\<D\>.P(D),D)$
    for some monic differential polynomial $P(D)$ of degree $n$. 
  \end{rem}

  \begin{lem}\label{cha:diff-modul-spectr}
  Let $L$, $P$ and $Q$ be
  differential polynomials, such that $L=QP$. Then we have an exact
  sequence of differential modules:
  \[\xymatrix{ 0\ar[r]& A\< D\>/A\< D\> Q\ar[r]^i & A\< D\>/A\< D\>
      L\ar[r]^p & A\< D\>/A\< D\> P\ar[r] & 0}\]
  where the maps $i$ and $p$ are defined as follows: for a
  differential polynomial $R$, $i(\bar{R})=\overline{RP}$ and
  $p(\bar{R})=\bar{R}$.
\end{lem}

\begin{proof}
  See \cite[Section~3.5.6]{Chr}.
\end{proof}
  
\subsection{Spectrum associated to a differential module}\label{sec:spectr-assoc-diff-4}
\begin{hyp}
  From now on $A$ will be either a $k$-affinoid algebra associated to
  an affinoid domain of $\Ak$ or $\Hx$ for some $x\in \Ak$ not of type
  (1). Let $d$ be a bounded derivation on $A$. It is of the form $d=g(T)\dT$
  where $g(T)\in A$.
\end{hyp}

Let $(M,\nabla)$ be a differential module over $(A,d)$. In order to associate
to this differential module a spectrum we need to endow  it with a
structure of $k$-Banach space. For that, recall the following
proposition:

\begin{pro}
  There exists an equivalence of category between the category of
  finite Banach $A$-modules with bounded $A$-linear maps as morphisms
  and the category of finite $A$-modules with $A$-linear maps as morphisms.
\end{pro}

\begin{proof}
  See \cite[Proposition 2.1.9]{Ber}.
\end{proof}

This means that we can endow
$M$ with a structure of finite Banach $A$-module
isomorphic to $A^n$ equipped with the maximum norm, and any other
structure of finite Banach $A$-module on $M$ is equivalent to the previous one. This induces  a
structure of Banach $k$-space on $M$. As $\nabla $ satisfies the rule
\eqref{eq:1} and $d\in \Lk{A}$, we have $\nabla \in \Lk{M}$. The
spectrum associated to $(M,\nabla)$ is denoted by $\Sigma_{\nabla,
  k}(\Lk{M})$\footnote{Note that, since all the structures of finite
  Banach $A$-module on $M$ are equivalente, the spectrum does not
  depend on the choice of such structure.} (or just by $\Sigma_\nabla$ if the dependence is obvious from the
context).

Let $\varphi:(M,\nabla)\to (N,\nabla ')$ be a morphism of differential
modules. If we endow $M$ and $N$ with a structure of $k$-Banach space
(as above) then $\varphi$ is automatically an admissible\footnote{ Which
  means that: $M/\Ker\varphi$ endowed with quotient the
  topology is isomorphic  as $k$-Banach space to $\operatorname{Im}\varphi$.} bounded
$k$-linear map (see \cite[Proposition~2.1.10]{Ber}). In the case $\varphi$ is an isomorphism, then it induces  a
bi-bounded $k$-linear isomorphism and according to Lemma
\ref{sec:berk-spectr-theory-1} we have:
\begin{equation}
  \label{eq:3}
  \Sigma_{\nabla,k}(\Lk{M})=\Sigma_{\nabla ', k}(\Lk{N}).
\end{equation}
This prove the following proposition:

\begin{pro}
  The spectrum of a connection is an invariant by isomorphisms of differential modules.
\end{pro}

\begin{pro}\label{4}
	Let $0\rightarrow (M_1,\nabla_1)\rightarrow (M,\nabla)\rightarrow (M_2,\nabla_2)\rightarrow 0$
	be an exact sequence of differential modules over $(A,d)$.\\Then we have:
	$\Sigma_{\nabla}(\Lk{M}) \subset\Sigma_{\nabla_1}(\Lk{M_1})\cup \Sigma_{\nabla_2}(\Lk{M_2})$, with equality if\\ ${\Sigma_{\nabla_1}(\Lk{M_1})\cap \Sigma_{\nabla_2}(\Lk{M_2})=\emptyset}$.
\end{pro}	
\begin{proof}
	As $M_1$, $M_2$ and $M$ are free $A$-modules, the
        sequence: \[\xymatrix{0\ar[r] & M_1\ar[r]^f &M\ar[r]^g&
            M_2\ar[r]& 0\\}\] splits. Hence, we have $M=M_1\oplus M_2$
        where $f$ is the inclusion of $M_1$ into $M$ and $g$ is the
        projection of $M$ onto $M_2$. Let $p_1$ be the projection of
        $M$ onto $M_1$ and $i_2$ be the inclusion of
        $M_ 2$ into $M$. As both $f$ and $g$ are morphisms of differential
        modules, we have $\nabla(M_1)\subset M_1$,  $\nabla_1=p_1\nabla f$ and $\nabla_2=g\nabla i_2$. 
	By Lemma \ref{10} and Remark \ref{16} we obtain the result.		
      \end{proof}

\begin{rem}\label{18}
	We maintain the assumption of Lemma \ref{4}. If in addition we have an other exact sequence of the form: \[0\rightarrow (M_2,\nabla_2)\rightarrow (M,\nabla)\rightarrow (M_1,\nabla_1)\rightarrow 0,\] 
	then the equality holds, this is a consequence of Remark
        \ref{16}. Indeed, If $\nabla\ot 1-1\ot T(x)$ is invertible,
        then the first exact sequence shows that $\nabla_1\ot 1-1\ot
        T(x)$ is left invertible and $\nabla_2\ot 1-1\ot T(x)$ is
        right invertible, the second exact sequence to $\nabla_2\ot 1-1\ot
        T(x)$ is left invertible and $\nabla_1\ot 1-1\ot T(x)$ is
        right invertible. Therefore, both of $\nabla_1\ot 1-1\ot T(x)$
        and $\nabla_2\ot 1-1\ot T(x)$ are invertible. Hence, we obtain $\Sigma_{\nabla_1}\cup\Sigma_{\nabla_2}\subset\Sigma_{\nabla}$.
      \end{rem}

      \begin{rem}\label{sec:spectr-assoc-diff}
        If moreover we have $M=M_1\oplus M_2$ as differential
        modules then we have ${\Sigma_{\nabla}=\Sigma_{\nabla_1}\cup \Sigma_{\nabla_2}}$.
      \end{rem}

\begin{rem}
  Set notation as in Proposition~\ref{4}. We suppose that $A=\Hx$ for
  some point $x\in \Ak$ not of type (1). For the spectral semi-norm it
  is know (see \cite[Lemma 6.2.8]{Ked}) that we have:
  \[\nsp{\nabla}=\max\{\nsp{\nabla_1}, \nsp{\nabla_2}\}.\]
\end{rem}

We say that a differential module $(M,\nabla)$ over
$(A,d)$ of rank $n$ is trivial if it isomorphic 
to  $(A^n,d)$ as a differential module.

\begin{lem}\label{sec:spectr-assoc-diff-3}
  Let $(M,\nabla)$ be a differential module over a differential field
  $(K,d)$ of rank $n$. If the $k$-vector space $\Ker{\nabla}$ has
  dimension equal to $n$, then $(M,\nabla)$ is a trivial differential module.
\end{lem}

\begin{proof}
  See \cite[Proposition~3.5.3]{Chr}.
\end{proof}

\begin{cor}
We suppose that $A=\Hx$ for some $x\in\Ak$ not of type (1). Let
$(M,\nabla)$ be a differential module over $(A,d)$. If $\text{dim}(\Ker\nabla)=n$, then $\Sigma_\nabla(\Lk{M})=\Sigma_d(\Lk{A})$.
\end{cor}

\begin{proof}
By Lemma~\ref{sec:spectr-assoc-diff-3} there exists
$\{e_1,\dots,e_n\}$  a basis of $M$ as an $A$-module for which
$\nabla$ satisfies the rule:
\[\nabla \begin{pmatrix} f_1\\ \vdots \\ f_{n} \end{pmatrix}
		= \begin{pmatrix} df_1\\ \vdots \\
                  df_{n} \end{pmatrix}.\]
By induction and  Remark \ref{sec:spectr-assoc-diff}
              we obtain the result. 
            \end{proof}

\begin{lem}\label{19}
	We suppose that $k$ is algebraically closed. Let $(M,\nabla)$
        be a differential module over $(A,d)$ such that $G\in
        \cM_n(k)$(cf.~\eqref{eq:1}) and $\{a_1,\cdots, a_N\}$ is the
        set of the eigenvalues of $G$. Then we have an isomorphisme of
        differential modules:
	\[(M,\nabla)\simeq (\bigoplus\limits_{1\leq i\leq
            N}\bigoplus\limits_{1\leq j\leq N_i} A\< D\>/(D-a_i)^{n_{i,j}},D) \]
	where the $n_{i,j}$ are positive integers such that
        $\sum_{j=1}^{N_i}n_{i,j}$ is the multiplicity of $a_i$  and  $\sum_{i,j}n_{i,j}=n$.
\end{lem}
\begin{proof}
	Consequence of the Jordan reduction.
\end{proof}	
	\begin{lem}\label{20}
		Let $(M,\nabla)$ be the differential
                module over $(A,d)$  associated to the differential polynomial
                $(D-a)^n$, where $a\in k$. The spectrum of $\nabla$ is
                $\Sigma_\nabla(\Lk{M})=a+\Sigma_d(\Lk{A})$ (the image of $\Sigma_d$ by
                the polynomial $T+a$).
              \end{lem} 

              \begin{proof}
		By Lemma~\ref{cha:diff-modul-spectr}, we have the exact sequences:
		\[0\to(A\<D\>/(D-a)^{n-1},D)\to (A\<D\>/(D-a)^n,D)\to (A\<D\>/(D-a),D)\to 0\] 
		and 
		\[0\to(A\<D\>/(D-a),D)\to (A\<D\>/(D-a)^n,D)\to (A\<D\>/(D-a)^{n-1},D)\to 0\]
		By induction and Remark
                \ref{18}, we have $\Sigma_D=\Sigma_{d+a}$. By Lemma
                \ref{sec:berk-spectr-theory}, we obtain
                $\Sigma_D=a+\Sigma_d$.
	\end{proof}

	\begin{pro}\label{6}
		We suppose that $k$ is algebraically closed. Let $(M,\nabla)$ be a differential module over $(A,d)$
                such that:
		\[\nabla \begin{pmatrix} f_1\\ \vdots \\ f_n \end{pmatrix}
		= \begin{pmatrix} df_1\\ \vdots \\ df_n \end{pmatrix}+G 
		\begin{pmatrix} f_1\\ \vdots \\ f_n \end{pmatrix},\]
		with $G\in \cM_n(k)$. The spectrum of $\nabla$ is
                $\Sigma_{\nabla}=\bigcup_{i=1}^{N}(a_i+\Sigma_d)$, where
                $\{a_1,\dots, a_N\}$ are  the eigenvalues of $G$.
               
	\end{pro}
	
	\begin{proof}
		Using the decomposition of Lemma \ref{19}, \ref{20}  and Remark \ref{sec:spectr-assoc-diff}, we obtain the result. 
		
              \end{proof}

              \begin{rem}
                This claim shows that the spectrum of a connection
                depends highly on the choice of the derivation $d$.
              \end{rem}

              \begin{nota}
                From now on we will fix $S$ to be the coordinate
                function of the analytic domain where the linear
                differential equation is defined and $T$ to be the
                coordinate function on $\Ak$ (for the computation of
                the spectrum).
              \end{nota}

\begin{lem}\label{13}
We assume that $k$ is algebraically closed. Let $\omega$ be the real
positive number introduced in \eqref{eq:7}.
  \begin{itemize}                
  \item Let $X$ be a connected affinoid domain as in
                \eqref{eq:2} and set $r=\min\limits_{0\leq i\leq \mu} r_i$. The operator norm of $(\d)^n$ as an element of $\Lk{\cO(X)}$ satisfies:
		\[\|(\d)^n\|_{\Lk{\cO(X)}}=\frac{|n!|}{r^n},\qqq \|\d\|_{Sp,\Lk{\cO(X)}}=\frac{\omega}{r}. \]

\item Let $x\in \Ak$ be a point of type (2), (3) or (4). The operator norm of $(\d)^n$ as an element of $\Lk{\Hx}$ satisfies:
		\[\|(\d)^n\|_{\Lk{\Hx}}=\frac{|n!|}{r(x)^n},\qqq \|\d\|_{Sp,\Lk{\Hx}}=\frac{\omega}{r(x)}. \]
\end{itemize}
\end{lem}

\begin{proof}
See \cite[Lemma~4.4.1]{and}. 
\end{proof}

\begin{rem}\label{sec:spectr-assoc-diff-2}
  We maintain the assumption that $k$ is algebraically closed. Let $X$
  be an affinoid domain of $\Ak$. Then $X=\bigcup_{i=1}^\mu X_i$, where
  $X_i$ are connected affinoid domains and $X_i\cap X_j=\emptyset$ for
  $i\not=j$. We have $\cO(X)=\bigoplus\limits_{i=1}^\mu\cO(X_i)$. As $\d$
  stabilises each Banach space of the direct sum, we have:
  \[\Nsp{\Lk{\cO(X)}}{\d}=\max_{0\leq i\leq \mu}\Nsp{\Lk{\cO(X_i)}}{\d}.\]
\end{rem}

Let $\Omega\in E(k)$ and let $X$ be an affinoid domain of
$\Ak$. Let $d=f(S)\d$ be a derivation defined on $\cO(X)$. We
can extend it to a derivation 
$d_{\Omega}=f(S)\d$ definedon$\cO(X_\Omega)$. The derivation
$d_{\Omega}$ is the image of $d\ot 1$ by the morphism $\Lk{\cO(X)}\ct_k
\Omega\to \LL{\Omega}{\cO(X_\Omega)}$ defined in Lemma \ref{3}.
  \begin{lem}\label{sec:spectr-assoc-diff-1}
 Let $\pi_{\Omega/k}:X_{\Omega}\to X$ be the canonical projection.  We have: \[\pi_{\Omega/k}(\Sigma_{d_\Omega,\Omega}(\LL{\Omega}{\cO(X_{\Omega})})\subset \Sigma_{d,k}(\Lk{\cO(X)}).\]
\end{lem}

\begin{proof}
By Lemma \ref{3} and \ref{sec:berk-spectr-theory-1} we have
$\Sigma_{d_\Omega,\Omega}(\LL{\Omega}{\cO(X_{\Omega})})\subset
\Sigma_{d\ot 1,\Omega}(\Lk{\cO(X)}\ct_k\Omega)$. Since\\
$\Sigma_{d\ot
  1,\Omega}(\Lk{\cO(X)\ct_k\Omega})=\pi_{\Omega/k}\-1(\Sigma_{d,k}(\Lk{\cO(X)}))$ (see
\cite[Proposition~7.1.6]{Ber}), we obtain the result.
\end{proof}

              \section{Main result}\label{sec:stat-main-result}

This section is divided in two parts. The first one is for the
       computation of the spectrum of $\d$, the second to state and prove the
       main result which is the computation of the spectrum associated
       to a linear differential equation with constant coefficients.  

       \begin{hyp}
         In this section we will suppose that $k$ is algebraically
       closed.
       \end{hyp}
       \subsection{The spectrum of $\d$ defined on several domains}\label{sec:spectrum-d-over}

       Let $X$ be an affinoid domain of $\Ak$ and $x\in\Ak$ be a point
       of type (2), (3) or (4). In this part we compute the spectrum of $\d$ as a derivation of
       $A=\cO(X)$ or $\Hx$  as
       an element of $\Lk{A}$. We
       treat the case of positive residual characteristic separately. We
       will also distinguish the case where $X$ is a closed disk from the case
       where it is a connected affinoid subdomain, and the case where
       $x$ is point of type (4) from the others.

	\subsubsection{The case of positive residual characteristic  }
       We suppose that $\crk=p>0$. In this case $\omega=|p|^{\frac{1}{p-1}}$.

        \begin{pro}\label{5}
		The spectrum of $\d$ as an element of
                $\Lk{\fdisf{c}{r}}$ is: \[{\Sigma_{\d}(\Lk{\fdisf{c}{r}})=\disf{0}{\frac{\omega}{r}}}.\]
	\end{pro} 		
	
	\begin{proof}	
		We set $A=\fdisf{c}{r}$ and $d=\d$. We prove firstly this claim for a field $k$ that is
                 spherically complete  and satisfies
                $|k|=\R+$. By Lemma \ref{13} the spectral norm of
                $d$ is equal to $\nsp{d}=\frac{\omega}{r}$. By
                Theorem~\ref{sec:defin-basic-propr} we have
                $\Sigma_d\subset \disf{0}{\frac{\omega}{r}}$.
                We prove now that $\disf{0}{\frac{\omega}{r}}\subset
                \Sigma_d$.

                Let $x\in \disf{0}{\frac{\omega}{r}}\cap
                k$. Then \[d\ot 1-1\ot T(x)=(d-a)\ot 1\] where
                $T(x)=a\in k$. The element $d\ot 1-1\ot T(x)$ is
                invertible in $\Lk{A}\ct \Hx$ if and only if $d-a$ is
                invertible in $\Lk{A}$ \cite[Lemma~7.1.7]{Ber}. 

                If $|a|<\frac{\omega}{r}$, then
                $\exp(a(S-c))=\sum_{n\in\NN}(\frac{a^n}{n!})(S-c)^n
                $ exists and it is an element of $A$. Hence,
                ${\exp(a(S-c))\in \ker (d-a)}$,
                in particular $d-a$ is not invertible.
                Consequently, $\diso{0}{\frac{\omega}{r}}\cap k\subset \Sigma_d$.
		
		Now we suppose that $|a|=\frac{\omega}{r}$. We prove
                that $d-a$ is not surjective.\\ Let
                ${g(S)=\sum_{n\in\NN}b_n(S-c)^n\in A}$.
                If there exists $f(S)=\sum_{n\in \NN}a_n(S-c)^n\in A$ such that
                ${(d-a)f=g}$,
                then for each $n\in\NN$ we have:
                \begin{equation}
                  \label{eq:8}
                  a_n=\frac{(\sum_{i=0}^{n-1}i!b_ia^{n-1-i})+a^na_0}{n!}.
                \end{equation}

		We now construct a series $g\in A$ such that its
                antecendent $f$ does not converge on the closed disk
                $D^+(c,r)$. Let $\alpha,\beta\in k$, such that $|\alpha|=r$ and
                $|\beta|=|p|^{1/2}$. For $n\in\NN$ we set:
		
		\[ b_n= \begin{cases}
		\frac{\beta^l}{\alpha^{p^l-1}} & \text{if } n=p^l-1
                \text{ with } l\in\NN \\
		0 & \text{otherwise}
		\end{cases} \] 
		Then, $|b_n|r^n$ is either $0$ or $|p|^{\log_p(n+1)}$ and $g\in A$. If we suppose that there exists $f\in A$ such that $(d-a)f=g$ then we have:
		
		\[\forall l\in \NN; \qqq a_{p^l}=\frac{a^{p^l-1}}{(p^l)!}[\sum_{j=0}^l\frac{(p^j-1)!\beta^j}{a^{p^j-1}\alpha^{p^j-1}}+aa_0].\]
		As $|p^l!|=\omega^{p^l-1}$ (cf. \cite[p. 51]{Dwo94}) we have:	
		\[|a_{p^l}|=\frac{1}{r^{p^l-1}}|\sum_{j=0}^{l}\frac{(p^j-1)!\beta^j}{a^{p^j-1}\alpha^{p^j-1} }+aa_0|.\] 
		Since
                $|\frac{(p^j-1)!\beta^j}{a^{p^j-1}\alpha^{p^j-1}}|=|p|^{-j/2}$,
                we have:	
		
		\[|\sum_{j=0}^{l}\frac{(p^j-1)!\beta^j}{a^{p^j-1}\alpha^{p^j-1}}|=\max_{0\leq j\leq l}|p|^{-j/2}=|p|^{-l/2},\]
		therefore $|a_{p^l}|r^{p^l}\xrightarrow{k\rightarrow +\infty}+\infty$, which proves that the power series $f$ is not in $A$ and this is a contradiction.
		Hence, $\disf{0}{\frac{\omega}{r}}\cap k \subset
                \Sigma_d$. As the points of type (1) are dense
                in $\disf{0}{\frac{\omega}{r}}$ and $\Sigma_d$ is
                compact, we deduce that $\disf{0}{\frac{\omega}{r}}\subset\Sigma_d$.
		
		Let us now consider an arbitrary field $k$. Let
                $\Omega\in E(k)$  algebraically closed, spherically
                complete such that $|\Omega|=\R+$. We denote 
                $A_\Omega=\cO(X_\Omega)$ and $d_\Omega=\d$ the
                derivationon$A_\Omega$. From above, we have
                $\Sigma_{d_\Omega}=\discf{\Omega}{0}{\frac{\omega}{r}}$,
                then
                $\Pro{k}{\Omega}(\Sigma_{d_\Omega})=\disf{0}{\frac{\omega}{r}}$. By
                Lemma \ref{sec:spectr-assoc-diff-1} we have
                $\disf{0}{\frac{\omega}{r}}=\Pro{k}{\Omega}(\Sigma_{d_\Omega})\subset
                \Sigma_d$. As $\nsp{d}=\frac{\omega}{r}$, we obtain $\Sigma_d=\disf{0}{\frac{\omega}{r}}$.

	\end{proof}
	
	\begin{rem}
		The statement holds even if the field $k$ is not
                algebraically closed. Indeed, we did not use this
                assumption.
              \end{rem}

	\begin{pro}\label{11}
		Let $X=\disf{c_0}{r_0}\setminus\bigcup_{i=1}^\mu
          \diso{c_i}{r_i}$ be a connected affinoid domain of $\Ak$
          different from the closed disk. The spectrum of $\d$ as an element of
                $\Lk{\cO(X)}$
                is: \[\Sigma_{\d}(\Lk{\cO(X)})=\disf{0}{\frac{\omega}{\Min_{0\leq
                          i\leq \mu} r_i}
                      }.\]
\end{pro}

\begin{proof}
               We set $A=\cO(X)$ and $d=\d$. The spectral norm of $d$ is equal to
                $\nsp{d}=\frac{\omega}{\min\limits_{0\leq i\leq \mu}r_i}$
                (cf. Lemma \ref{13}), which implies $\Sigma_d\subset
            \disf{0}{\frac{\omega}{\min\limits_{0\leq i\leq \mu}r_i}}$. Now, let $x\in
                \disf{0}{\frac{\omega}{\min\limits_{0\leq i\leq \mu}r_i}}$. We set
                    $A_{\Hx}=\cO(X_{\Hx})$ and $d_{\h{x}}=\d:\;
                    A_{\Hx}\to A_{\Hx}$. From Lemma
                    \ref{3} we have the bounded morphism:
                    \[\Lk{A}\Ct{k}\h{x}\rightarrow\LL{\h{x}}{A_{\h{x}}}.\]
                   The image of $d$ by this morphism is the derivation
                    $d_{\h{x}}$. By the Mittag-Leffler decomposition
                    \cite[Proposition~2.2.6]{Van}, we have:
                    \[\cO(X_{\Hx})=\bigoplus\limits_{i=1}^n\{\sum_{j\in
                        \NN^*} \dfrac{a_{ij}}{(S-c_i)^j}|\; a_{ij}\in \Hx,\;
                      \lim\limits_{j\to +\infty}|a_{ij}|r_i^{-j}= 0 \}\oplus \fdiscf{\Hx}{c_0}{r_0}.\]
		Each Banach space of the direct sum above is stable under $d_{\Hx}$.\\
		We set $F_i=\{\sum_{j\in \NN^*}
                \dfrac{a_{ij}}{(S-c_i)^j}|\; a_{ij}\in \Hx,\; \lim\limits_{j\to +\infty}|a_{ij}|r_i^{-j}=0 \}$, and $d_i=d_{\Hx|_{F_i}}$. By Lemma \ref{10}, $\Sigma_{d_\Hx}=\bigcup \Sigma_{d_i}$.
		Let $i_0>0$ be the index such that $r_{i_0}=\min\limits_{0\leq
                  i\leq \mu}r_i$. We will prove that $d_{i_0}-T(x)$ is
                not surjective. Indeed, let $g(S)=\sum_{n\in
                  \NN^*}\frac{b_n}{(S-c_{i_0})^{n}}\in F_{i_0}$,  if
                there exists $f(S)=\sum_{n\in
                  \NN^*}\frac{a_n}{(S-c_{i_0})^{n}}\in F_{i_0}$ such
                that $(d_{i_0}-T(x))f(S)=g(S)$, then for each
                $n\in\NN^*$ we have:
		
		\[a_n=\frac{(n-1)!}{(-T(x))^n}\sum_{i=1}^n\frac{(-T(x))^{i-1}}{(i-1)!}b_i.\]
                We choose $g(S)=\frac{1}{S-c_{i_0}}$, in this case
                $a_n=\frac{(n-1)!}{(-T(x))^n}$ and
                $|a_n|=\frac{|(n-1)!|}{|T(x)|^n}$. As
                $|T(x)|\leq\frac{\omega}{r_{i_0}}$, the sequence $|a_n|r_{i_0}^{-n}$
                diverges. We obtain contradiction since $f\in F_{i_0}$. Hence, $d_{\Hx}- T(x)$ is not invertible,
                which implies that $d\ot 1-1\ot T(x)$ is not
                invertible
                and we obtain the result.

              \end{proof}

        \begin{cor}\label{sec:case-posit-resid}
          Let $X$ be an affinoid domain of $\Ak$. The spectrum of $\d$ as an element of
                $\Lk{\cO(X)}$ is: \[\Sigma_{\d}(\Lk{\cO(X)})=\disf{0}{\nsp{\d}}.\]

        \end{cor}

        \begin{proof}
          In this case we may write $X=\bigcup_{i=1}^uX_i$, where the $X_i$ are
          connected affinoid domain of $\Ak$ such that $X_i\cap
          X_j=\emptyset$ for $i\not= j$. We have:
          \[\cO(X)=\bigoplus\limits_{i=1}^u\cO(X_i).\]
          Each Banach space of the direct sum above is stable under
          $d$, we denote by $d_i$ the restriction of $d$ to
          $\cO(X_i)$. We have $\nsp{d}=\max\limits_{1\leq i\leq u}d_i$
          (cf. Remark~\ref{sec:spectr-assoc-diff-2}). By  Lemma
          \ref{10} and Proposition \ref{11} we have
          $\Sigma_d=\bigcup_{i=1}^u\disf{0}{\nsp{d_i}}=\disf{0}{\max_i\nsp{d_i}}$. Hence,
            we obtain the result. 
        \end{proof}

	\begin{pro}\label{sec:case-posit-char}
		Let $x\in \Ak$ be a point of type (2), (3) or (4). The spectrum of $\d$ as an element of
                $\Lk{\Hx}$ is: \[\Sigma_{\d}(\Lk{\Hx})=\disf{0}{\frac{\omega}{r(x)}}.\]
Where $r(x)$ is the value defined in Definition~\ref{sec:berkovich-line-3}.
	\end{pro}
	
	\begin{proof}
          We set $d=\d$. We distinguish three cases: 
		\begin{itemize}
			\item\textbf{$x$ is point a of type (2):} Let
                          $c\in k$ such that $x=x_{c,r(x)}$. By Proposition \ref{15}, we have:\[\Hx=E\oplus \fdisf{c}{r(x)}.  \]
			As both  $E$ and $\fdisf{c}{r(x)}$ are stable
                        under $d$,  by Lemma \ref{10}  we have\\
                        ${\Sigma_d=\Sigma_{d_{|_{E}}}\cup
                          \Sigma_{d_{|_{\fdisf{c}{r(x)}  }}}}$. By
                        Proposition \ref{5}
                        ${\Sigma_{d_{|_{\fdisf{c}{r(x)}
                              }}}=\disf{0}{\frac{\omega}{r(x)}}}$. \\Since
                        ${\nsp{d}=\frac{\omega}{r(x)}}$ (cf. Lemma \ref{13}), then $\Sigma_d=\disf{0}{\frac{\omega}{r(x)}}$.\\
			
			\item\textbf{$x$ is point a of type (3):} Let
                          $c\in k$ such that $x=x_{c,r(x)}$. In this
                          case ${\Hx=\fcouf{c}{r(x)}{r(x)}}$. By
                          Proposition~\ref{11} we obtain the result.\\

			\item\textbf{$x$ is point a of type (4):} By
                          Proposition \ref{14} we have $\Hx\ct_k
                          \Hx\simeq \fdiscf{\Hx}{S(x)}{r(x)}$.
                          From Lemma \ref{3} we have the bounded
                          morphism:
                          \[\Lk{\Hx}\ct_k\Hx\to
                            \LL{\Hx}{\fdiscf{\Hx}{S(x)}{r(x)}}.\]
                          The image of $d\ot 1$ by this morphism is the derivation
                          $d_{\Hx}=\d:\fdiscf{\Hx}{S(x)}{r(x)}\to \fdiscf{\Hx}{S(x)}{r(x)}$.
                         
                          From Proposition \ref{5} we have
                          ${\Sigma_{d_{\Hx}}=\discf{\Hx}{0}{\frac{\omega}{r(x)}}}$,
                          then $\Pro{k}{\Hx}(\Sigma_{d_{\Hx}})=\disf{0}{\frac{\omega}{r(x)}}$. By
                          Lemma \ref{sec:spectr-assoc-diff-1} we  have
                          $\disf{0}{\frac{\omega}{r(x)}}\subset
                          \Sigma_d$. Since
                          $\nsp{d}=\frac{\omega}{r(x)}$ (cf. Lemma
                          \ref{13}), we obtain
                          $\Sigma_d=\disf{0}{\frac{\omega}{r(x)}}$
                          (cf. Theorem \ref{sec:defin-basic-propr}).
		\end{itemize}
		
	\end{proof}

        \subsubsection{ The case of  residual characteristic zero}

        We suppose that $\crk=0$.

        \begin{pro}\label{sec:case-zero-residual}
         The spectrum of $\d$ as an element of
                $\Lk{\fdisf{c}{r}}$ is: \[\Sigma_{\d}(\Lk{\fdisf{c}{r}})=\overline{\diso{0}{\frac{1}{r}}}\]
         (the topological closure of $\diso{0}{\frac{1}{r}}$).
             
          \end{pro}

          \begin{proof}
           We set $A=\fdisf{c}{r}$ and $d=\d$. The spectral norm of $d$ is equal to
            $\nsp{d}=\frac{1}{r}$ (cf. Lemma~\ref{13}), which implies that $\Sigma_d\subset
            \disf{0}{\frac{1}{r}}$ (cf. Theorem~\ref{sec:defin-basic-propr}).
            For $a\in \diso{0}{\frac{1}{r}}\cap k$, $d-a$ is not
                    injective, therefore $a\in \Sigma_d$. Let $x\in
                    \diso{0}{\frac{1}{r}}\setminus k$.

                    We set
                    $A_{\h{x}}=A\Ct{k}\h{x}=\fdiscf{\h{x}}{c}{r}$ and
                    $d_{\h{x}}=\d:A_{\Hx}\to A_{\Hx}$. From Lemma
                    \ref{3} we have the bounded morphism:
                    \[\Lk{A}\Ct{k}\h{x}\rightarrow\LL{\h{x}}{A_{\h{x}}}\]
The derivation
                    $d_{\h{x}}$ is the image of $d\ot 1$ by this morphism. As $|T(x)(S-c)|<1$, $f=\exp
                    (T(x)(S-c))$ exists and it is an element of
                    $A_{\Hx}$ . As  $f\in
                    \Ker(d_{\h{x}}-T(x))$,  $d_{\h{x}}-T(x)$ is
                    not invertible. Therefore $d\ot 1-1\ot T(x)$ is
                    not invertible which is equivalent to saying that
                    $x\in \Sigma_d$. By  compactness of the
                    spectrum we have $\overline{\diso{0}{\frac{1}{r}}} \subset
                    \Sigma_d$. In order to end the proof, we need to prove firstly the statement 
                    for the case where $k$ is trivially valued.
\begin{itemize}
\item \textbf{Trivially valued case:}
 We need to distinguish  two cases:
            \begin{itemize}
              
              \item $\b{r\not=1:}$ In this case we have
                $\disf{0}{\frac{1}{r}}=\overline{\diso{0}{\frac{1}{r}}}$,
                hence $\Sigma_d=\overline{\diso{0}{\frac{1}{r}}}$.
                
                \item $\b{r=1:}$ In this case we have
                  $\fdisf{c}{1}=k[S-c]$ equipped with the trivial
                  valuation, and $\Lk{k[S-c]}$ is the
                  $k$-algebra of all $k$-linear maps equipped with the
                  trivial norm (i.e. $\nor{\phi}=1$ for all $\phi\in {\Lk{k[S-c]}\setminus\{0\}}$). Let $a\in
                    k\setminus\{0\}$. Since the power series
                    $\exp(a(S-c))=\sum_{n\in
                      \NN}\frac{a^n}{n!}(S-c)^n$ does not converge in
                    $\fdisf{c}{1}$, the operator
                    $d-a:k[S-c]\rightarrow k[S-c]$ is injective. It is
                    alos surjective. Indeed, let
                    $g(S)=\sum_{n=0}^mb_n(S-c)^n\in \fdisf{c}{1}$. The
                    polynomial $f(S)=\sum_{n=0}^{m}a_n(S-c)^n\in
                    \fdisf{c}{1}$ that its coefficients verifies 
\[a_n=\frac{-a^{n-1}}{n!}\sum_{i=n}^m i!b_i a^{-i}.\]
                     for all $0\leq n\leq m$, satisfies
                     $(d-a)f=g$. Hence, $d-a$ is invertible in
                     $\Lk{k[S-c]}$. Since the norm is trivial on
                     $\Lk{k[S-c]}$, we have 
                    $\nsp{(d-a)\-1}\-1=1$. Therefore, by Lemma \ref{sec:defin-basic-propr-1}
                   , for all $x\in
                    \diso{a}{1}$ the element $d\ot 1-1\ot T(x)$ is
                    invertible. Consequently, for all $a\in k\setminus\{0\}$ the disk $D^-(a,1)$ is not meeting the spectrum $\Sigma_d$. This means that $\Sigma_d$ is contained in $D^+(0,1)\setminus\bigcup_{a\in k\setminus\{0\}}D^-(a,1)=[0,x_{0,1}]$. Since $[0,x_{0,1}]=\overline{D^-(0,1)}$ we have $\Sigma_d=\overline{\diso{0}{1}}$.

                      \end{itemize}
    \item \textbf{Non-trivially valued case:}
               We need to
            distinguish  two cases:
            \begin{itemize}
           \item $\b{r\notin|k^*|:}$ In this case we have
                $\disf{0}{\frac{1}{r}}=\overline{\diso{0}{\frac{1}{r}}}$,
                hence $\Sigma_d=\overline{\diso{0}{\frac{1}{r}}}$.

            \item $\b{r\in |k^*|:}$ 
              We can reduce our case to
              $r=1$. Indeed, there exists an isomorphism of $k$-Banach
              algebras
              \[\fdisf{c}{r}\to \fdisf{c}{1},\]
              that associates to $S-c$ the element $\alpha(S-c)$, with $\alpha\in k$
              and $|\alpha|=r$. This induce an isomorphism of
              $k$-Banach algebras
              \[\Lk{\fdisf{c}{r}}\to \Lk{\fdisf{c}{1}},\]
              which associates to $d:\fdisf{c}{r}\to \fdisf{c}{r}$
              the derivation $\fra{\alpha}\cdot\d:\fdisf{c}{1}\to
              \fdisf{c}{1}$. By Lemmas \ref{sec:berk-spectr-theory-1}
              and~\ref{sec:berk-spectr-theory} we obtain
              \[\Sigma_d(\Lk{\fdisf{c}{r}})=\fra{\alpha}\Sigma_{\d}(\Lk{\fdisf{c}{1}}).\]

              We now suppose that $r=1$. Let $k'$ be a maximal (for the order given by the
              inclusion) trivially valued field
              included in $k$ (which exists by  Zorn's Lemma). As
              $k$ is algebraically closed then so is
              $k'$. The complete residual field of $x_{0,1}\in\A{k'}$ is
              $\h{x_{0,1}}=k'(S)$ endowed with the trivial valuation,
              the maximality of $k'$ implies that $\h{x_{0,1}}$ can
              not be included in $k$, therefore $k\not\in
              E(\h{x_{0,1}})$. By
              Proposition~\ref{sec:berkovich-line}, we obtain
                $\piro{k}{k'}\-1(x_{0,1})=\{x_{0,1}\}$. We set $d'=\d$ as an element of
                $\LL{k'}{\fdiscf{k'}{c}{1}}$.  We know by
                \cite[Proposition~7.1.6]{Ber} that the spectrum of
                $d'\ot 1$, as an element of
                $\LL{k'}{\fdiscf{k'}{c}{1}}\Ct{k'} k$, satisfies
                $\Sigma_{d'\ot 1}=\piro{k}{k'}^{-1}(\Sigma_{d'})$,
                by the result above we have $\piro{k}{k'}\-1(\Sigma_{d'})=\overline{\diso{0}{1}}$. From  Lemma \ref{3}  we have a
              bounded morphism $\LL{k'}{\fdiscf{k'}{c}{1}}\Ct{k'} k
              \to \Lk{\fdisf{c}{1}}$, the image of $d'\ot 1$  by this
              morphism is
                $d$.  Therefore,
                $\Sigma_d\subset\piro{k}{k'}\-1(\Sigma_{d'})=\overline{\diso{0}{1}}$. Then
                we obtain the result.

              \end{itemize}

              \end{itemize}

                      \end{proof}

              \begin{pro}\label{sec:case-zero-residual-1}
                	              Let $X=\disf{c_0}{r_0}\setminus\bigcup_{i=1}^{\mu}\diso{c_i}{r_i}$
              be a connected affinoid domain of $\Ak$ different from
              the closed disk.
 The spectrum of $\d$ as an element of
                $\Lk{\cO(X)}$
                is: \[\Sigma_{\d}(\Lk{\cO(X)})=\disf{0}{\fra{\Min_{0\leq
                        i\leq \mu}r_i}}.\]
              \end{pro}
              \begin{proof}
                Here $\omega=1$ (cf. \eqref{eq:7}). The proof is
                 the same as in Proposition~\ref{11}.
              \end{proof}

              \begin{cor}
                Let $X$ be an affinoid domain of $\Ak$ which does
                not contain a closed disk as a connected component. The spectrum of $\d$ as an element of
                $\Lk{\cO(X)}$ is: \[\Sigma_{\d}(\Lk{\cO(X)})=\disf{0}{\nsp{\d}}.\]

              \end{cor}

              \begin{rem}
                Let $X$ be an affinoid domain of $\Ak$.
                Then $X=Y\cup D$ with $Y\cap D=\emptyset$, where $Y$ is an affinoid domain as
                in the corollary above and $D$ is a disjoint union of
                disks. We set $d_Y=\d_{|_{\cO(Y)}}$ and
                $d_D=\d_{|_{\cO(D)}}$. If $\nsp{d_Y}\geq\nsp{d_D}$ then
                $\Sigma_{\d}=\disf{0}{\nsp{d_Y}}$. Otherwise, $\Sigma_{\d}= \overline{\diso{0}{\nsp{d_D}}}$.
              \end{rem}
              
	\begin{pro}\label{sec:case-resid-char}
          	Let $x\in \Ak$ be a point of type (2), (3) . The spectrum of $\d$ as an element of
                $\Lk{\Hx}$
                is: \[\Sigma_{\d}(\Lk{\Hx})=\disf{0}{\frac{1}{r(x)}}.\]
                  Where $r(x)$ is the value defined in Definition~\ref{sec:berkovich-line-3}.
	\end{pro}
	
	\begin{proof}
		We set $d=\d$. We distinguish two cases:
		\begin{itemize}
			\item\textbf{$x$ is point of type (2):} Let
                          $c\in k$ such that $x=x_{c,r(x)}$. By
                          Proposition~\ref{15} we have \[\Hx=F\oplus
                            \fcouf{c}{r(x)}{r(x)}  \]
                          where,
                          \[F:=\widehat{\bigoplus}_{\alpha\in
                    \tilde{k}\setminus\{0\}}\{\sum_{i\in\NN^*} \frac{a_{\alpha i}}{(T-c+\gamma\alpha)^i}|\; a_{\alpha i}\in
                  k,\; \lim\limits_{i\to +\infty}|a_{\alpha i}|r^{-i}= 0 \}.\] 
                          
                          We use the same arguments as in Proposition \ref{sec:case-posit-char}. 
			\item\textbf{$x$ is point of type (3):} Let
                          $c\in k$ such that $x=x_{c,r(x)}$. In this
                          case $\Hx=\fcouf{c}{r(x)}{r(x)}$, by Proposition~\ref{sec:case-zero-residual-1} we conclude.\\

		\end{itemize}
		
              \end{proof}

              \begin{pro}\label{sec:case-resid-char-1}
                Let $x\in\Ak$ be a point of type (4). The spectrum of
                $\d$ as an element of $\Lk{\Hx}$ is: \[\Sigma_{\d}(\Lk{\Hx})=\overline{\diso{0}{\frac{1}{r(x)}}}.\]Where $r(x)$ is the value defined in Definition~\ref{sec:berkovich-line-3}.              \end{pro}

              \begin{proof}
             We set $d=\d$. By
                          Proposition \ref{14} we have $\Hx\ct_k
                          \Hx\simeq \fdiscf{\Hx}{S(x)}{r(x)}$.
                         From Lemma \ref{3} we have the bounded
                          morphism:
                          \[\Lk{\Hx}\ct_k\Hx\to
                            \LL{\Hx}{\fdiscf{\Hx}{S(x)}{r(x)}}\]
                          which
                          associates to $d\ot 1$ the derivation
                          $d_{\Hx}=\d:\fdiscf{\Hx}{S(x)}{r(x)}\to \fdiscf{\Hx}{S(x)}{r(x)}$.
                          From Proposition \ref{sec:case-zero-residual} we have
                          ${\Sigma_{d_{\Hx}}=\overline{\disco{\Hx}{0}{\frac{1}{r(x)}}}}$,
                          hence $\Pro{k}{\Hx}(\Sigma_{d_{\Hx}})=\overline{\diso{0}{\frac{1}{r(x)}}}$. By
                          Lemma \ref{sec:spectr-assoc-diff-1} we  have
                          ${\overline{\diso{0}{\frac{1}{r(x)}}}\subset
                          \Sigma_d}$. From now on we set $r=r(x)$. Since
                        $\nsp{d}=\frac{1}{r}$ (cf. Lemma
                        \ref{sec:spectr-assoc-diff-2}) and
                        $\Sigma_d\subset \disf{0}{\nsp{d}}$
                        (cf. Theorem \ref{sec:defin-basic-propr}), in
                        order to prove the statement
                        it is enough to show that for
                        all $a\in k$ such that $|a|=\frac{1}{r}$, we have
                        $\diso{a}{\frac{1}{r}}\subset\Ak\setminus\Sigma_d$. Let
                        $a\in k$ such that $|a|=\frac{1}{r}$. The
                        restriction of $d-a$ to the normed $k$-algebra $k[S]$ is a bijective
                        bounded map $d-a:k[S]\to k[S]$ with respect to
                        the restriction of 
                        $|.|_x$. We set $\phi=(d-a)_{|_{k[S]}}$. As $\Hx$ is
                        the completion of $k[S]$ with respect to
                        $|.|_x$ (cf. Lemma~\ref{sec:berkovich-line-1}), to prove that it
                        extends to an isomorphism, it is enough to show
                        that $\phi\-1 :k[S]\to k[S]$ is a  bounded
                        $k$-linear map. A family of closed
                        disks $\{\disf{c_l}{r_l}\}_{l\in I}$ is called embedded
                        if the set of index $I$ is endowed
with total order $\leq$ and for $i\leq j$ we have
$\disf{c_i}{r_i}\subset\disf{c_j}{r_j}$. Since $x$ is a point of type
(4), then there exists a family of embedded disks
$\{\disf{c_l}{r_l}\}_{l\in I}$ such that $\bigcap_{l\in I}\disf{c_l}{r_l}=\{x\}$.
If we consider $d-a$ as an element of
                      $\Lk{\fdisf{c_l}{r_l}}$, then it is invertible
                      (cf. Proposition \ref{sec:case-zero-residual})
                      and its restriction to $k[S]$ coincides with
                      $\phi$ as $k$-linear map. 
                      Let $f(S)=\sum_{i\in\NN}a_i(S-c_l)^i$ and
                      $g(S)=\sum_{i\in\NN}b_i(S-c_l)^i$ be two
                      elements of $\fdisf{c_l}{r_l}$ such that $(d-a)f=g$. Using the same
                      induction to obtain the equation \eqref{eq:8} we
                      obtain:
                      for all $n\in\NN$
                      \begin{equation}
                        \label{eq:9}
                        a_n=\frac{-a^{n-1}}{n!}\sum_{i\geq n}i!b_ia^{-i}.
                      \end{equation}
                      Hence,

                      \begin{equation}
                        \label{eq:10}
                        |a_n|=\frac{1}{r^{n-1}}|\sum_{i\geq
                          n}i!b_i
                        a^{-i}|\leq\frac{1}{r^{n-1}}\max\limits_{i\geq
                          n}|b_i|r^i\leq r\max\limits_{i\geq
                          n}|b_i|r^{i-n}\leq r\max\limits_{i\geq n}|b_i|r_l^{i-n}\leq
                        \frac{r}{r_l^{n}}\max\limits_{i\geq
                          n}|b_i|r_l^i
                      \end{equation}
%
                      Therefore,
                      \begin{equation}
                        \label{eq:11}
                        |a_n|r_l^n\leq r\max\limits_{i\geq
                          n}|b_i|r_l^i.                     
                      \end{equation}
                      Consequently,
                      \[ |f|_{x_{c_l,r_l}}\leq r|g|_{x_{c_l,r_l}}\]
                      In the special case where $f$ and $g$ are in
                      $k[S]$, then $f=\phi\-1(g)$ and we have for all
                      $l\in\NN$:
                      \[ |\phi\-1(g)|_{x_{c_l,r_l}}\leq
                        r|g|_{x_{c_l,r_l}}.\]
                      Hence,
                      \[|\phi\-1(g)|_x=\inf\limits_{l\in\NN}|\phi\-1(g)|_{x_{c_l,r_l}}\leq \inf\limits_{l\in\NN}r|g|_{x_{c_l,r_l}}=r|g|_x.
                      \]
                      This means that $\phi\-1$ is  bounded, hence
                        $d-a$ is invertible in $\Lk{\Hx}$ and
                       $\nor{(d-a)\-1}\leq r$, %
                       hence $\nsp{(d-a)\-1}\leq r$. Since $\nsp{(d-a)\-1}\-1$ is
                       the radius of the biggest disk centred in $a$
                       contained in $\Ak\setminus \Sigma_d$ (cf. Lemma
                      \ref{sec:defin-basic-propr-1}), we obtain 
                       $\diso{a}{\frac{1}{r}}\subset\Ak\setminus\Sigma_d$.

              \end{proof}

  \subsection{Spectrum of a linear differential equation with constant coefficients}\label{sec:spectr-diff-equat}

        Let $X$ be an affinoid domain of $\Ak$ and $x\in X$ a
        point of type (2), (3) or (4). We set here $A=\cO(X)$ or $\Hx$ and
        $d=\d$. Recall that a linear differential equation with constant
        coefficients is a differential module $(M,\nabla)$ over $(A,d)$ associated to a
        differential polynomial
        ${P(D)=g_0+g_1D+\dots+g_{\nu-1}D^{\nu-1}+D^\nu}$ with $g_i \in k$, or in an
        equivalent way there exists a basis for which the matrix $G$
        of the rule \eqref{eq:1} has constant coefficients (i.e $G\in
        \cM_\nu(k)$). Here we compute the spectrum of $\nabla$ as an
        element of $\Lk{M}$ (cf. Section \ref{sec:spectr-assoc-diff-4}).

        \begin{Theo}\label{sec:spectr-diff-equat-1}
          Let $X$          
          be a connected affinoid domain of $\Ak$. We set here
          $A=\cO(X)$. Let $(M,\nabla)$ be a defferential module over
          $(A,d)$ such that the matrix $G$ of the rule \eqref{eq:1}
          has constant entries (i.e. $G\in\cM_\nu(k)$), and let
          $\{a_1,\cdots, a_N\}$ be the set of eigenvalues of $G$. Then we have:
            \begin{itemize}
            \item If $X=\disf{c_0}{r_0}$, \[\Sigma_{\nabla,k}(\Lk{M})=
                \begin{cases}
                  \Bigcup_{i=1}^{N}\disf{a_i}{\frac{\omega}{r_0}}&
                  \text{If } \crk>0\\
                  &\\
                  \Bigcup_{i=1}^{N}\overline{\diso{a_i}{\frac{1}{r_0}}}&
                  \text{If } \crk=0\\
                \end{cases}.\]
            \item If
              $X=\disf{c_0}{r_0}\setminus\bigcup_{i=1}^\mu\diso{c_i}{r_i}$
              with $\mu \geq 1$,
              \[\Sigma_{\nabla,k}(\Lk{M})=\bigcup_{i=1}^{N}\disf{a_i}{\frac{\omega}{\min\limits_{0\leq
                      i\leq \mu}r_i}}.\]
          \end{itemize}
 Where $\omega$ is the positive real number introduced in \eqref{eq:7}.

        \end{Theo}
        \begin{proof}
       By Propositions~\ref{5},~\ref{11},~\ref{sec:case-zero-residual},
       \ref{sec:case-zero-residual-1} and \ref{6} we obtain the result.    
        \end{proof}

        \begin{Theo}\label{sec:spectr-diff-equat-2}
          Let $x\in\Ak$          
          be a point of type (2), (3) or (4). We set here
          $A=\Hx$. Let $(M,\nabla)$ be a defferential module over
          $(A,d)$ such taht the matrix $G$ of the rule \eqref{eq:1}
          has constant entries (i.e. $G\in\cM_\nu(k)$), and let
          $\{a_1,\cdots, a_N\}$ be the set of eigenvalues of $G$. Then we have:
            \begin{itemize}
           \item If
              $x$ is a point of type (2) or (3),
              \[\Sigma_{\nabla,k}(\Lk{M})=\bigcup_{i=1}^{N}\disf{a_i}{\frac{\omega}{r(x)}}.\]
               \item If $x$ is a point of type (4), \[\Sigma_{\nabla,k}(\Lk{M})=
                \begin{cases}
                  \Bigcup_{i=1}^{N}\disf{a_i}{\frac{\omega}{r(x)}}&
                  \text{If } \crk>0\\
                  &\\
                  \Bigcup_{i=1}^{N}\overline{\diso{a_i}{\frac{1}{r(x)}}}&
                  \text{If } \crk=0\\
                \end{cases}.\]
           
          \end{itemize}
 Where $\omega$ is the positive real number introduced in \eqref{eq:7}.

\end{Theo}

\begin{proof}
 By
 Propositions~\ref{sec:case-posit-char},~\ref{sec:case-resid-char},~\ref{sec:case-resid-char-1}
 and \ref{6} we obtain the result. 
\end{proof}
        
\begin{rem}\label{sec:spectr-line-diff-1}
  Notice that since the spectrum of $\nabla$ is independant of the
  choice of the basis, if $G'$ is another associated matrix to the differential
  module $(M,\nabla)$ with constant entries. Then  the set of 
  eigenvalues  $\{a'_1,\cdots, a'_{N'}\}$ of $G'$ can not be arbitrary, namely it must satisfy: for each
  $a'_i$ there exists $a_j$ such that $a'_i$  belongs to the connected
  componente of the spectrum containing $a_j$. 
\end{rem}

        \begin{rem}\label{sec:spectr-line-diff}
As mentioned in the introduction, if we consider the differential polynomial $P(d)$ as an
          element of $\Lk{A}$, then its spectrum
          $\Sigma_{P(d)}=P(\Sigma_{d})$
          (cf. Lemma~\ref{sec:berk-spectr-theory}) which is in general diffrent
          from the spectrum of the associated connexion.
        \end{rem}

        \section{Variation of the spectrum}   \label{sec:spectrum-variation}

In this section, we will discuss about the behaviour of the
        spectrum of $(M,\nabla)$ over $(\Hx,d)$, when we vary $x$ over
        $[x_1,x_2]\subset \Ak$, where $x_1$ and $x_2$ are points of
        type (2), (3) or (4). For that we need to define a topology over $K(\Ak)$
        the set of nonempty compact subsets of $\Ak$. Note that, in
        the case $\Ak$ is metrisable, we can endow $K(\Ak)$ with a
        metric called Hausdorff metric. However, in general $\Ak$ is
        not metrisable. Indeed, it is metrizable
        if and only if $\tilde{k}$ is countable. In the first
        part of the section, we will introduce  the topology on
        $K(\cT)$ (set of nonempty compact subset of a Hausdorff
        topological space $\cT$), that coincides with the topology
        induced by the Hausdorff metric in the metrizable case. In the second, we will prove that the
        variation of the spectrum of a differential equation with
        constant coefficients is left continuous.


        \subsection{The topology on $K(\cT)$}
        Let $\cT$ be a Hausdorff topological space, we will denote by $K(\cT)$
        the set of nonempty compact subset of $\cT$. Recall that, in
        the case where $\cT$ is metrizable. For an associated metric $\mathscr{M}$, the
        respective Hausdorff metric $\MM_H$ definedon
        $K(\cT)$ is given as
        follows. Let $\Sigma$, $\Sigma'\in K(\cT)$
        \begin{equation}
          \label{eq:13}
          \MM_H(\Sigma,\Sigma')=\max\{ \sup_{\beta\in
            \Sigma'}\inf_{\alpha\in
            \Sigma}\MM(\alpha,\beta),\sup_{\alpha\in
            \Sigma}\inf_{\beta\in \Sigma'}\MM(\alpha,\beta)\}. 
        \end{equation}
        We introduce below a topology on $K(\cT)$ for an arbitrary
        Hausdorff topological space $\cT$, that coincides with the topology
        induced by the Hausdorff metric in the metrizable case.
        
        \paragraph{The topology on $K(\cT)$:} Let $U$ be an open of
        $\cT$ and $\{U_i\}_{i\in I}$ be a finite open cover of $U$. We
        set:
        \begin{equation}
          \label{eq:12}
          (U,\{U_i\}_{i\in I})=\{ \Sigma\in K(\cT)|\;  \Sigma\subset U,\; \Sigma\cap
          U_i\not= \emptyset\; \forall i\}.
        \end{equation}
        The family of sets of this form is stable under finite
        intersection. Indeed, we have:
        \[(U,\{U_i\}_{i\in I})\cap (V,\{V_j\}_{j\in J})=(U\cap
          V,\{U_i\cap V\}_{i\in I}\cup \{V_j\cap U\}_{j\in J}).\]
        We endow $K(\cT)$ with the topology
        generated by this family of  sets.

        \begin{lem}
          The topological space $K(\cT)$ is Hausdorff.
        \end{lem}

        \begin{proof}
          Let $\Sigma$ and $\Sigma'$ be two compact subsets of $\cT$
          such that $\Sigma\ne\Sigma'$. We may assume that $\Sigma'\not\subset\Sigma$. Let
          $x\in\Sigma\setminus \Sigma'$. Since $\cT$ is a
          Hausdorff space, there exists an open
          neighbourhood of $U_x$ of $x$ and an open neighbourhood $U'$
          of $\Sigma'$, such that $U_x\cap
          U'=\emptyset$. Let $U$ be an open neighbourhood of $\Sigma$
          such that $U_x\subset U$. Then the open set $(U,\{U,U_x\})$
          (resp. $(U',\{U'\})$) is a
          neighbourhood of $\Sigma$ (resp. $\Sigma'$) in $K(\cT)$
          such that $(U,\{U,U_x\})\cap (U',\{U'\})=\emptyset$. 

        \end{proof}

        \begin{lem}
          Assume that $\cT$ is metrizable. The topology on
          $K(\cT)$ coincides with the topology induced by the
          Hausdorff metric.
        \end{lem}

        \begin{proof}
          Let $\MM$ be a metric associated to $\cT$. For $x\in \cT$
          and $r\in \R+^*$ we
          set \[ B_{\MM}(x,r):=\{y\in \cT|\; \MM(x,y)<r\}.\]
          For $\Sigma\in
          K(\cT)$ and $r\in\R+^*$, we set
          $B_{\MM_H}(\Sigma,r):=\{\Sigma'\in K(\cT)|\; \MM_H(\Sigma,\Sigma')<r\}$.
          Let $\Sigma\in K(\cT)$. To prove the statement, we need to show that for all $r\in \R+^*$ there
          exists an open neighbourhood $(U,\{U_i\}_{i\in I})$ of $\Sigma$ such that
          $(U,\{U_i\}_{i\in I})\subset B_{\MM_H}(\Sigma,r)$, and for
          each open neighbourhood $(U,U_i)_{i\in I}$ of $\Sigma$,
          there exists $r\in \R+^*$ such that
          $B_{\MM_H}(\Sigma,r)\subset(U,\{U_i\}_{i\in I})$.

          Let $r\in \R+^*$. Let $\{c_1,\cdots, c_m\}\subset
          \Sigma$ such that $\Sigma\subset
          \bigcup_{i=1}^mB_\MM(c_i,\frac{r}{3})$, we set
          $U=\bigcup_{i=1}^mB_\MM(c_i,\frac{r}{3})$. We claim that
          $(U,\{B_\MM(c_i,\frac{r}{3})\}_{i=1}^m)\subset
          B_{\MM_H}(\Sigma,r)$. Indeed, let $\Sigma'\in
          (U,\{B_\MM(c_i,\frac{r}{3})\}_{i=1}^m)$. As $\Sigma'\in U$, for
          all $y\in \Sigma'$ we have 
          $\min_{1\leq i\leq m}(\MM(y,c_i))<\frac{r}{3}$.
          Therefore,
          \[\sup_{y\in\Sigma'}\inf_{x\in \Sigma}\MM(y,x)\leq
            \frac{r}{3}<r.\]
          Since for each $c_i$ there exists $y\in \Sigma'$ such that
          $\MM(c_i,y)<\frac{r}{3}$, for each $x\in \Sigma$ there
          exists $y\in \Sigma'$ such that
          $\MM(x,y)<\frac{2r}{3}$. Indeed, there exists $c_i$ such that
          $x\in B(c_i, \frac{r}{3})$, therefore we have
          \[\MM(x,y)\leq \MM(c_i,y)+\MM(c_i,x)< \frac{2r}{3}.\]
          This implies
          \[\sup_{x\in \Sigma}\inf_{y\in \Sigma'}\MM(x,y)\leq\frac{2r}{3}<r.\]
          Consequently, $\MM_H(\Sigma,\Sigma')<r$.
          
          Now let $(U,\{U_i\}_{i\in I})$ be an open neighbourhood of
          $\Sigma$. Let $\alpha=\inf_{y\in \cT\setminus U}\inf_{x\in \Sigma}\MM(x,y)$,
          since $\Sigma\subsetneq U$ we have $\alpha\ne 0$. For each
          $1\leq i\leq m$, let $c_i\in \Sigma\cap U_i$. There exists
          $0<\beta<\alpha$ such that for all $0<r<\beta$ we have $B_{\MM}(c_i,r)\subset U_i$ for each $1\leq
          i\leq m$.
          
We claim that $B_{\MM_H}(\Sigma,r)\subset
          (U,\{U_i\}_{i=1}^{m})$. Indeed, let $\Sigma'\in
          B_{\MM_H}(\Sigma,r)$ this means that:
          \[\sup_{y\in \Sigma'}\inf_{x\in
              \Sigma}\MM(x,y)<r;\qqq \sup_{x\in
              \Sigma}\inf_{y\in \Sigma'}\MM(x,y)<r.\]
          The first inequality implies $\Sigma'\subset U$. The second
          implies that for each $c_i$ there exists $y\in \Sigma'$ such
          that ${\MM(c_i,y)<r}$. Hence, $\Sigma'\cap U_i=\emptyset$ for
          each $i$.
        \end{proof}

        \subsection{Variation of the spectrum}
        
        Let $X$ be an affinoid domain of $\Ak$. Let $(M,\nabla)$ be a differential module over
        $(\cO(X),\d)$ such that there exists a basis for which
        the associated matrix $G$ has constant entries. For a point $x\in X$
        not of type (1), the differential module $(M,\nabla)$ extends
        to a differential module $(M_x,\nabla_x)$ over $(\Hx,\d)$. In
         the corresponding basis of $(M_x,\nabla_x)$ the associated matrix is $G$.

        \begin{Theo}\label{sec:topology-kak}
        Suppose that $k$ is algebraically closed. Let
        $X=\disf{c_0}{r_0}\setminus\bigcup_{i=1}^\mu\diso{c_i}{r_i}$
        be a connected affinoid domain and $x\in X$ be a point of type
        (2), (3) or (4). Let
          $(M,\nabla)$ be a differential module over
          $(\cO(X),\d)$ such that there exists a basis for which
          the corresponding matrix $G$ has constant entries. We set:
          \[\Fonction{\Psi:
              [x,x_{c_0,r_0}]}{K(\Ak)}{y}{\Sigma_{\nabla_y}(\Lk{M_y})}.\]
          Then we have:
          \begin{itemize}
          \item for each $y\in [x,x_{c_0,r_0}]$, the restriction of
            $\Psi$ to $[x,y]$ is continuous at $y$.
          \item If $y\in [x,x_{c_0,r_0}]$ is a point of type (3),
            then $\Psi$ is continuous at $y$.
          \item If $\crk=0$ and $y\in [x,x_{c_0,r_0}]$ is a point of type
            (4), then $\Psi$ is continuous at $y$.   
          \end{itemize}
        \end{Theo}

        \begin{proof}
          Let $\{a_1,\cdots, a_N\}\subset k$ be the set of eigenvalues
          of $G$. We identifie $[x,x_{c_0,r_0}]$ with the interval
          $[r(x),r_0]$ by the map $y\mapsto r(y)$ (cf. Definition
          \ref{sec:berkovich-line-3}). Let $y\in [x,x_{c_0,r_0}]$. We set
              $\Sigma_y=\Sigma_{\nabla_y}(\Lk{M_y})$.  By Theorem
           \ref{sec:spectr-diff-equat-2},  for all $y'\in
              (x,x_{c_0,r_0}]$ ($y'$ can not be a point of type (4)) we have
            $\Sigma_{y'}=\bigcup_{i=1}^N\disf{a_i}{\frac{\omega}{r(y')}}$. Let
              $(U,\{U_i\}_{i\in I})$ be an open neighbourhood of
              $\Sigma_y$. Since $\Sigma_y$ is a finite union of closed
              disks, we may assume that $U$ is a finite union of  open
              disks. Let $R$ be the smallest radius of those
              disks.
          \begin{itemize}
          \item If $y=x$ then it is
            obvious that the restriction of $\Psi$ to $[x,y]=\{y\}$ is
            continuous at $y$. Now, we assume that $y\ne x$. Let
            $x_R\in (x,y)$ be the point with radius $r(x_R)=\max(r(x),\frac{\omega}{R})$.
           Then  for each $y'\in (x_R,y]$ we have
           $\frac{\omega}{r(y')}<R$. Hence, $\Sigma_{y'}\subset U$ for each $y'\in (x_R,y]$. 
              Since $\frac{\omega}{r(y)}\leq \frac{\omega}{r(y')}$ for
              each $y'\in (x_R,y]$, we have
              $\Sigma_y\subset\Sigma_{y'}$. Therefore 
              $\Sigma_{y'}\cap U_i\ne
              \emptyset$ for each $y'\in (x_R,y]$. Then we obtain
              $(x_R,y]\subset\Psi\-1((U,\{U_i\}_{i\in I}))$.
            \item Let $y\in [x,x_{c_0,r_0}]$ be point of type
              (3). Since $k$ is algebraically closed, $\omega\in |k|$.
              Therefore, $\frac{\omega}{r(y)}\not\in |k|$ and we have
              $\Sigma_y=\bigcup_{i=1}^N\disf{a_i}{\frac{\omega}{r(y)}}=\bigcup_{i=1}^N(\diso{a_i}{\frac{\omega}{r(y)}}\cup\{x_{a_i,\frac{\omega}{r(y)}}\})$.
              In order to
              prove the continuity at $y$ it is enough to prove that
              the restriction $\Psi$ to $[y,x_{c_0,r_0}]$ is
              continuous at $y$. Let $y'\in [y,x_{c_0,r_0}]$. Since
              $\frac{\omega}{r(y')}\leq \frac{\omega}{r(y)}$, we
              have $\Sigma_{y'}\subset\Sigma_y\subset U$. Now we
              need to show that there exists $x'\in (y,x_{c_0,r_0}]$ such that for all
              $y'\in [y,x')$ we have $\Sigma_{y'}\cap U_i\ne
              \emptyset$. Let $\alpha_1,\cdots,\alpha_m\in\Sigma_y$ such that $\alpha_i\in
              \Sigma_y\cap U_i$ for each $i$. Since $\Sigma_y$ is an
              affinoid domain we can assume that the $\alpha_i$ are
              either of type (1) or (3)\footnote{Note that, in the
                case where $k$ is not trivially valued, we may assume
                that the $\alpha_i$ are of type (1).}. For each
              $\alpha_i$ there exists $a_{j_i}$ such that $\alpha_i\in
              \diso{a_{j_i}}{\frac{\omega}{r(y)}}\cup\{ x_{a_{j_i},\frac{\omega}{r(y)}}\}\subset \Sigma_y$.

              If $\alpha_i$ is point of
              type (1), then $\alpha_i\in
              \diso{a_{j_i}}{\frac{\omega}{r(y)}}$. Since the open
              disks form a basis of neighbourhoods for the points of
              type (1), there exists
              $\diso{\alpha_i}{L_i}\subset\diso{a_{j_i}}{\frac{\omega}{r(y)}}\cap U_i$. We set $R_i=\max(|b_i-a_{j_i}|,L_i)$. Let  $x_i\in [y,x_{c_0,r_0}]$ with
              $r(x_i)=\min(r(x_{c_0,r_0}),\frac{\omega}{R_i})$.  Then for all $y'\in
              (y,x_i)$ we have
              $R_i<\frac{\omega}{r(y')}<\frac{\omega}{r(y)}$. This implies 
              $\diso{\alpha_i}{R_i}\subset\diso{a_{j_i}}{\frac{\omega}{r(y')}}\subset\Sigma_{y'}$ for each $y'\in
              [y,x_i)$. 

              Suppose now that $\alpha_i=x_{b_i,L_i}$ is
             point of type (3). If $\alpha_i\ne x_{a_{j_i},\frac{\omega}{r(y)}}$, we set
              $R_i=\max(|a_{j_i}-b_i|,L_i)$. Let $x_i\in
              (y,x_{c_0,r_0}]$ with $r(x_i)=\min(r(x_{c_0,r_0}),\frac{\omega}{R_i})$. Since $\frac{\omega}{r(y')}>R_i$ for each $y'\in
              [y,x_i)$, we have $\alpha_i\in \diso{a_{j_i}}{\frac{\omega}{r(y')}}\subset\Sigma_{y'}$. 
              If $\alpha_i=x_{a_{j_i},\frac{\omega}{r(y)}}$, since it
              is a point of type (3), the open annulus 
              $\couo{a_{j_i}}{L^1_i}{L^2_i}$ form a basis of
              neighbourhoods of $\alpha_i$. Therefore, there exists
              $\couo{a_{j_i}}{L^1_i}{L^2_i}\subset U_i$ containing
              $\alpha_i$. This implies that there exists
              $\couf{a_{j_i}}{R_i}{\frac{\omega}{r(y)}}\subset
              \disf{a_{j_i}}{\frac{\omega}{r(y)}}\cap U_i$. Let
              $x_i\in [y,x_{c_0,r_0}]$ with
              $r(x_i)=\min(r(x_{c_0,r_0}),\frac{\omega}{R_i})$. Then for all $y'\in
              (y,x_i)$ we have $R_i<
              \frac{\omega}{r(y')}<\frac{\omega}{r(y)}$. Therefore
              $\couf{a_{j_i}}{R_i}{\frac{\omega}{r(y)}}\cap\disf{a_{j_i}}{\frac{\omega}{r(y')}}=\couf{a_{j_i}}{R_i}{\frac{\omega}{r(y')}}$. Then
              we obtain $\Sigma_{y'}\cap U_i\ne \emptyset$ for each
              $y'\in [y,x_i)$.
Hence, for all
              $y'\in\bigcap_{i=1}^m[y,x_i)=[y,x_{i_0})$ we have
              $\Sigma_{y'}\cap U_i\ne \emptyset$.
            \item We assume that $\crk=0$. Let $y\in [x,x_{c_0,r_0}]$ be a point of type
              (4). Since $\Sigma_y=\bigcup_{i=1}^N(\diso{a_i}{\frac{\omega}{r(y)}}\cup\{x_{a_i,\frac{\omega}{r(y)}}\})$
              (cf. Theorem \ref{sec:spectr-diff-equat-2}), using the
              same arguments above we obtain the result.
          \end{itemize}
        \end{proof}

        \begin{rem} Set notations as in the proof of
          Theorem~\ref{sec:topology-kak}. 
          In the case where $y\in [x,x_{c_0,r_0})$ is a point of type
          (2), then the map $\Psi$ is never continuous at $y$. Indeed,
          since $y$ is of type (2), there exists
          $b\in
          (\disf{a_i}{\frac{\omega}{r(y)}}\setminus\diso{a_i}{\frac{\omega}{r(y)}})\cap
        k$. As for each $y'\in (y,x_{c_0,r_0})$ we have
        $\Sigma_{y'}\cap\diso{b}{\frac{\omega}{r(y)}}=\emptyset$,
        $\Sigma_{y'}\not\in (U\cup \diso{b}{\frac{\omega}{r(y)}},\{U_i\}_{i\in I}\cup
        \{\diso{b}{\frac{\omega}{r(y)}}\})$ for all $y'\in (y,x_{c_0,r_0})$.
        \end{rem}

\printbibliography

Tinhinane Amina, AZZOUZ \    \url{tinhinane-amina.azzouz@univ-grenoble-alpes.fr}

Univ. Grenoble Alpes, CNRS, Institut Fourier, F-38000 Grenoble, France

\end{document}